\documentclass[12 pt, leqno]{amsart}

\usepackage{amsmath, amsthm, amssymb, amsfonts}
\usepackage{a4wide}
\usepackage{tikz}
\usepackage[pagebackref=false]{hyperref}

\makeatletter
\@namedef{subjclassname@2020}{%
	\textup{2020} Mathematics Subject Classification}
\makeatother

\theoremstyle{plain}
\newtheorem{theorem}{Theorem}[section]
\newtheorem{thmi}{Theorem}
\newtheorem{cori}[thmi]{Corollary}

\newtheorem{lemma}[theorem]{Lemma}
\newtheorem{proposition}[theorem]{Proposition}
\newtheorem{corollary}[theorem]{Corollary}
\theoremstyle{definition}
\newtheorem*{question}{Question}

\newtheorem{definition}[theorem]{Definition}
\newtheorem*{convention}{Convention}
\newtheorem*{ack}{Acknowledgements}
\newtheorem{example}[theorem]{Example}

\newtheorem{remark}[theorem]{Remark}

\numberwithin{equation}{section}

\newcommand{\act}{\curvearrowright}

\newcommand{\id}{\mathord{\operatorname{id}}}
\newcommand{\Z}{\mathbb{Z}}
\newcommand{\Q}{\mathbb{Q}}
\newcommand{\R}{\mathbb{R}}
\newcommand{\N}{\mathbb{N}}

\newcommand{\la}{\left\langle}
\newcommand{\ra}{\right\rangle}
\newcommand{\Comm}{\mathrm{Comm}}

\newcommand{\rE}{\operatorname{E}}
\newcommand{\rV}{\operatorname{V}}

\newcommand{\Stab}{\operatorname{Stab}}

  \newcommand{\abs}[1]{\left\lvert #1\right\rvert}

\newcommand{\HN}{\mathcal H(\mathbb N)}
\newcommand{\HR}{\mathcal H(\mathcal R)}
\newcommand{\HUQ}{\mathcal H(\mathbb U_{\Q^+})}

\newcommand{\HUS}{\mathcal H(\mathbb U_S)}
\newcommand{\HX}{\mathcal H(X)}

\newcommand{\Gr}{\mathcal{G}}

\newcommand{\EG}{\rE(\Gr)}
\newcommand{\VG}{\rV(\Gr)}

\DeclareMathOperator{\diam}{\mathrm{diam}}
\DeclareMathOperator{\Iso}{\mathrm{Iso}}
\DeclareMathOperator{\dom}{\mathrm{dom}}
\DeclareMathOperator{\rng}{\mathrm{rng}}

\newcommand{\norm}[1]{\left\lVert #1 \right\rVert}

\newcommand{\inv}{^{-1}}

\newcommand{\U}{\mathbb{U}}
\newcommand{\UQ}{\mathbb{U}_{\mathbb{Q^+}}}
\newcommand{\UQone}{\mathbb{U}_{S}}
\newcommand{\US}{\UQone}

\newcommand{\autUQone}{\Iso(\mathbb{U}_{S})}

\begin{document}


\baselineskip=17pt


\author[P. Fima]{Pierre Fima}
\address{Pierre Fima
\newline
Universit\'e de Paris, Sorbonne Universit\'e, CNRS, Institut de Math\'ematiques de 
Jussieu-Paris Rive Gauche, F-75013, Paris, France.}
\email{pierre.fima@imj-prg.fr}

\author[F. Le Ma\^itre]{Fran\c{c}ois Le Ma\^itre}

\address{Fran\c{c}ois Le Ma\^itre
\newline
Universit\'e de Paris, Sorbonne Universit\'e, CNRS, Institut de Math\'ematiques de 
Jussieu-Paris Rive Gauche, F-75013, Paris, France.}
  \email{f.lemaitre@math.univ-paris-diderot.fr}

\author[J. Melleray]{Julien Melleray}
\address{Julien Melleray
\newline
Universit\'e de Lyon, Universit\'e Claude Bernard - Lyon 1, CNRS UMR 5208, Institut Camille Jordan, 43 Boulevard du 11 novembre 1928, 69622 Villeurbanne Cedex, France}
\email{melleray@math.univ-lyon1.fr}

  \author[S. Moon]{Soyoung Moon}

\address{Soyoung Moon
\newline
Universit\'e de Bourgogne, Institut Math\'ematiques de Bourgogne, UMR 5584 CNRS, 
BP 47870, 21078 Dijon cedex 
France}
\email{soyoung.moon@u-bourgogne.fr}

\title{Homogeneous actions on Urysohn spaces.}
\begin{abstract}
We show that many countable groups acting on trees, including free products of infinite countable groups and surface groups, are isomorphic to dense subgroups of isometry groups of bounded Urysohn spaces. This extends previous results of the first and last author with Y. Stalder on dense subgroups of the automorphism group of the random graph.  
In the unbounded case, we also show that every free product of infinite countable groups arises as a dense subgroup of the isometry group of the rational Urysohn space. 
\end{abstract}



\maketitle

\tableofcontents
\section{Introduction}

There is an extensive, and still growing, literature concerning countable groups $G$ admitting faithful actions on $\N$ which are \emph{highly transitive}, that is, such that for any partial bijection $\sigma$ with finite domain there exists $g \in G$ such that $g\cdot x= \sigma(x)$ for all $x \in \dom(\sigma)$. Such groups are  called highly transitive groups.
It is clear that the group $\mathfrak S_{(\infty)}$ consisting of finitely supported bijections of $\N$ is highly transitive. For a finitely generated example, one can consider the group of permutations of $\Z$ which are translations up to a finite set, or more generally Houghton groups. 

The first explicit examples of  highly transitive non-amenable groups are free groups $\mathbb{F}_n$ for $2\leq n\leq +\infty$, as was shown in 1976 by T.P. McDonough \cite{zbMATH03558079} (see also the work of J.D. Dixon in \cite{zbMATH04105195}). The case of a general free product has been studied by A.A.W. Glass and S.H. McCleary in \cite{zbMATH04181610} and later settled by S.V. Gunhouse \cite{zbMATH04193986} and independently by K.K. Hickin \cite{zbMATH00120147}: a free product $\Gamma=\Gamma_1*\Gamma_2$ of two non-trivial countable groups $\Gamma_i$, $i=1,2$, is highly transitive if and only if $\Gamma$ is i.c.c.\footnote{A group is i.c.c. (infinite conjugacy classes) when all the conjugacy classes of nonidentity elements are infinite.}, which is equivalent to ask that at least one  the $\Gamma_i$ has cardinal at least $3$. In the last few years, many new examples of highly transitive groups have been discovered:  non elementary hyperbolic groups with trivial finite radical (V.V. Chaynikov \cite{zbMATH06120601}), surface groups (D. Kitroser \cite{zbMATH06204044}), Out$(\mathbb{F}_n)$ for $n\geq 4$ (S. Garian and Y. Glasner \cite{zbMATH06190856}), many groups acting on trees 
(P. Fima, S. Moon and Y. Stalder \cite{zbMATH06503094}), and finally acylindrically hyperbolic group with trivial finite radical (M. Hull and D. Osin \cite{zbMATH06673802}).

Note that the class of highly transitive countable groups consists of all countable groups which are isomorphic to a dense subgroup of $\mathfrak S_\infty$, the group of all permutations of $\N$. It makes sense to ask a similar question for other groups.

\begin{question}Given a Polish group $G$, what can be said about the class of all its countable dense subgroups? 
\end{question}

This question is particularly interesting when $G$ is the isometry group of a \emph{homogeneous} countable metric space, i.e. a countable metric space with the property that any partial isometry with finite domain extend to a global isometry. Indeed, if $(X,d)$ is a countable homogeneous metric space, then a countable subgroup $\Gamma$ of $\Iso(X)$ is dense\footnote{Note that since $X$ is countable, $\Iso(X)$ is endowed with the topology of pointwise convergence, viewing $X$ as a \emph{discrete} space.} if and only if its natural action on $X$ satisfies the following definition. 

\begin{definition}
An action of a countable group $\Gamma$ on a homogeneous metric space $(X,d)$ is called \textbf{homogenous} if every partial isometry of $X$ with finite domain extends to an element of $\Gamma$. 
\end{definition}

If $(X,d)$ is a homogeneous metric space, we denote by $\HX$ the class of all countable groups which admit a \emph{faithful} homogeneous isometric action on $X$. Note that the set  $\N$ endowed with the discrete metric is a homogeneous metric space, and its isometry group is $\mathfrak S_\infty$, so $\HN$ denotes the class of highly transitive groups.
In this paper, we are studying the class $\mathcal H(X)$ where $X$ is one of the \emph{countable Urysohn spaces}.

To define countable Urysohn spaces, one first fixes a countable set $S \subset [0,+\infty[$, containing $0$, and considers the class of all $S$-metric spaces, that is, the class of all metric spaces whose distance takes its values in $S$. Under a suitable assumption on $S$ (see the next section), there exists a generic object in this class, which is characterized as being the unique $S$-metric space which contains an isometric copy of each finite $S$-metric space (\emph{universality}), and such that any isometry between finite subsets extends to the whole space (\emph{ultrahomogeneity}). We call this space the $S$-Urysohn space (when $S=\Q$, this space was built by Urysohn in the seminal paper \cite{zbMATH02582889}) and denote it by $\U_S$. For instance, when $S=\{0,1\}$, $\U_S$ is simply a countable infinite set with the discrete metric; for $S=\{0,1,2\}$, $\U_S$ may be identified with Erd\"os--Rad\'o's random graph $\mathcal R$ equipped with the path metric.

Little is known about the class $\HR$. It contains $\mathbb{F}_2$ (H.D. Macpherson \cite{zbMATH03987259}), a locally finite group (M. Bhattacharjee and H.D. Macpherson \cite{zbMATH02186801}), many groups acting on trees such as  free products of any two infinite countable groups (P. Fima, S.Moon and Y. Stalder \cite{FMS16}). A result of S. Solecki (whose proof appears in a paper by C. Rosendal \cite{Rosendal2011}) yields that the class $\HUS$ always contains a locally finite group. Our main result relies on the following notion, which was introduced in \cite{zbMATH06503094}.

\begin{definition}\label{def: HCF intro}
A subgroup $\Sigma$ of a countable group $\Gamma$ is \textbf{highly core-free} if whenever $F$ is a finite subset of $\Gamma$, we may find $\gamma\in \Gamma$ such that the map 
\begin{align*}
\Sigma\times F&\to \Gamma\\
(\sigma,f)&\mapsto \sigma \gamma f
\end{align*}
is injective and its range $\Sigma \gamma F$ is disjoint from $F$. 
\end{definition}

Easy examples of highly core-free subgroups are provided by the trivial group and any finite subgroup of an i.c.c. group. 

\begin{thmi} [{see Thm. \ref{Trees}}]
Let $\Gamma$ be a countable group acting on a non-trivial tree $\mathcal{T}$. If each edge stabilizer is highly core-free in the two corresponding vertex stabilizers, then $\Gamma\in\HUS$ for all bounded distance set $S$.
\end{thmi}




\begin{cori}[{see Cor. \ref{cor: amalgamated examples} and Cor. \ref{cor: examples HNN}}]
Let $S$ be a bounded distance set  and $\Gamma_1,\Gamma_2,H$ be countably infinite groups with a common subgroup $\Sigma$. Let also $\theta\,:\,\Sigma\rightarrow H$ be an injective group homomorphism. The following holds.
\begin{enumerate}
\item If $\Sigma$ is highly core-free in both $\Gamma_1,\Gamma_2$ then $\Gamma_1\underset{\Sigma}{*}\Gamma_2\in\HUS$. In particular:
\begin{enumerate}
\item For any infinite countable groups $\Gamma_1$, $\Gamma_2$ one has $\Gamma_1*\Gamma_2\in\HUS$.
\item If $\mathcal S_g$ is a closed, orientable surface of genus $g>1$, then $\pi_1(\mathcal S_g)\in\HUS$.
 \end{enumerate}
 \item If $\Sigma$ and $\theta(\Sigma)$ are both highly core-free in $H$ then ${\rm HNN}(H,\Sigma,\theta)\in\HUS$.
\end{enumerate}
\end{cori}

We also study unbounded Urysohn spaces, for instance when $S=\Q^+$. In that case, we obtain the following result (see Corollary \ref{CorUnbounded}).

\begin{thmi}\label{thmi: homogeneous on UQ}
For any infinite countable groups $\Gamma_1$ and $\Gamma_2$ one has $\Gamma_1*\Gamma_2\in\HUQ$.
\end{thmi}

The method of proof finds its roots in a paper of Dixon \cite{zbMATH04105195}, and is an adaptation to the context of $S$-Urysohn spaces of techniques of P. Fima, S. Moon and Y. Stalder \cite{zbMATH06152250,zbMATH06503094,FMS16}. Let us describe briefly the argument in the case of a free product of infinite groups $\Gamma \ast \Lambda$. One first proves that any countable group admits a ``sufficiently free'' and ``sufficiently rich'' action on $\mathbb U_S$
(notions that are made clear below; this is easier to do in the case when $S$ is bounded, explaining why we have stronger results in that case). Then one uses a Baire category argument: start from two sufficiently free, sufficiently rich actions $\pi_1,\pi_2$ of $\Gamma, \Lambda$ respectively. Then, for any $\alpha \in \Iso(\U_S)$, one can consider the action $\pi_\alpha$ of $\Gamma \ast \Lambda$ which coincides with $\pi_1$ on $\Gamma$ and with $\alpha \pi_2 \alpha^{-1}$ on $\Lambda$; we then prove that $\{\alpha \colon \pi_\alpha \text{ is faithful }\}$ and 
$\{\alpha \colon \pi_\alpha(\Gamma \ast \Lambda) \text{ is dense} \}$ are both dense $G_\delta$ in $\Iso(\U_S)$, thus the intersection $O$ of these two sets is nonempty; and for any $\alpha \in O$ the subgroup $\pi_\alpha(\Gamma \ast \Lambda)$ is  dense in $\Iso(\U_S)$ and isomorphic to $\Gamma \ast \Lambda$. 

When $S$ is bounded, this basic strategy can be employed, with some technical modifications, to cover the case of amalgamated free products over highly core-free subgroups, and thus surface groups, as well as the case of free products where one of the factors is finite and the other infinite. Unfortunately, while the basic structure and underlying ideas are the same in all these proofs, technical aspects involving the triangle inequality differ, forcing us to write down several times some very similar arguments.

Looking at the results discussed above, one question in particular begs to be answered: what happens to the class $\HUS$ when $S$ varies? Do these classes all coincide, can there be nontrivial inclusions, etc.? This question remains largely open; however, we do manage to establish in the last section the following result. 

\begin{thmi}[{see Thm. \ref{t:anexample}}]
The group $\mathfrak S_{(\infty)}$ of all finitary permutations of $\N$ does not belong to $\HUS$ unless $|S|=2$, that is, unless $\Iso(\U_S) \cong \mathfrak S_\infty$. 
\end{thmi}

Finally, consider the Urysohn space $(\mathbb U,d)$, which can be obtained as the metric completion of $\UQ$. Then $\Iso(\U)$ is a Polish group for the topology of pointwise convergence induced by the metric $d$, and a result of Cameron-Vershik yields that $\Iso(\UQ)$ is a dense subgroup of $\Iso(\U)$ \cite{zbMATH05068543}. Theorem \ref{thmi: homogeneous on UQ} thus yields that every free product of infinite countable groups can be densely embedded in $\Iso(\U)$. It is however unclear whether every countable dense subgroup of $\Iso(\mathbb U)$ belongs to the class $\HUQ$. The same question can be asked in the bounded case. 

The paper is organized as follows: we first introduce some basic facts about group actions on metric spaces, and develop the machinery we need to deal with the bounded case. We then establish the theorems mentioned above in that case. Once that work is completed, we turn to the unbounded case. While similar in spirit, this case requires some additional work, including the construction of some isometric actions of countable groups on countable metric spaces with interesting combinatorial properties. Finally, we prove that $\mathfrak S_{(\infty)}$ does not belong to $\HUS$ unless $|S|=2$.

\begin{ack}
	The authors are deeply grateful to the referee for their detailed reports
	which led to many substantial improvements of the paper.
	
	F.\ Le Maître and J.\ Melleray's research was partially supported by ANR project GAMME (ANR-14-CE25-0004) and
	ANR project AGRUME (ANR-17-CE40-0026).
\end{ack}

\section{Preliminaries}

\begin{definition}
Consider a countable set $S\subseteq [0,+\infty[$ containing $0$ and at least another element. We say that $S$ is an \textbf{unbounded distance set} if $S$ is a subsemigroup of the additive semigroup $[0,+\infty[$, meaning that 
for all $s,t\in S$, we have $s+t\in S$. 

We say that $S$ is a \textbf{bounded distance set} if $S$ has a maximum $M=\max(S)>0$ with
\[
\label{SetSM}
\forall s,t\in S \quad \min(s+t,M) \in S .
\]
\end{definition}

Given a bounded or unbounded distance set $S$, we say that a metric space $(X,d)$ is an \textbf{$S$-metric space} if the metric $d$ takes values in $S$.

\begin{example}\label{ex: graphs as metric spaces}
Every graph can be viewed as a $\{0,1,2\}$-metric space by letting $d(x,y)= 1$ when there is an edge between $x$ and $y$, $d(x,y)=0$ if $x=y$ and $d(x,y)=2$ if $x\neq y$ and there is no edge between $x$ and $y$. Conversely every $\{0,1,2\}$-metric space can be seen as a graph equipped with the above metric. 
\end{example}
A map between two metric spaces $(X,d_X)$ and $(Y,d_Y)$ is an \textbf{isometry} if it is surjective and for all $x_1,x_2\in X$ we have $d_Y(f(x_1),f(x_2))=d_X(x_1,x_2)$. We denote by $\Iso(X,d_X)$ or simply $\Iso(X)$ the group of isometries $X\to X$. 

A \textbf{partial isometry} between $X$ and $Y$ is a  map $\varphi \colon \dom \varphi\subseteq X\to \rng\varphi\subseteq Y$ which is an isometry for the restrictions of $d_X$ and $d_Y$ to the  sets $\dom\varphi$ and $\rng\varphi$ respectively. 

A partial isometry is \textbf{finite} if its domain is. We denote by $P_f(X)$ the set of finite partial isometries between $X$ and $X$.

Throughout the paper, we will use the following notation : by $A \Subset X$ we mean $A\subseteq X$ and $A$ is finite. 

For any bounded or unbounded distance set $S$, the \emph{$S$-Urysohn space $\mathbb U_S$} is the unique, up to isometry, countable $S$-valued metric space
which has the \textbf{extension property}: given a finite $S$-metric space $(X,d)$, if $(Y,d)$ is another finite $S$-valued metric space containing $(X,d)$, any isometric embedding $\rho: X\to \mathbb U_S$ extends to an isometric embedding $\tilde \rho: Y\to \mathbb U_S$.  

The \emph{back-and-forth argument} allows one to prove that if $(X,d_X)$ and $(Y,d_Y)$ are two countable $S$-metric spaces satisfying the extension property, then any finite partial isometry $\varphi:\dom \varphi\Subset X\to \rng\varphi\Subset Y$ extends to an isometry $X\to Y$. 
Applying this to the empty map, one gets that the $S$-Urysohn space is indeed unique up to isometry. 

Before tackling the case $S=\Q\cap[0,+\infty[$, we will be focusing on bounded metric spaces, since our constructions work best in that case. Thus we make the following

\begin{convention}
Until section \ref{sec: unbounded}, we assume that $S \subseteq [0,1]$ and $0,1 \in S$. We thus have a bounded distance set $S$ satisfying 
\begin{equation}\label{SetS}
\forall s,t\in S \quad \min(s+t,1) \in S .
\end{equation}
\end{convention}
\begin{remark}
We make the convention that our bounded $S$-metric spaces have diameter at most $1$ for the sake of simplicity. This does not affect the strength of our results because we are interested in isometry groups, which remain the same when multiplying the metric by a constant.
\end{remark}

\begin{remark}
The $\{0,1\}$-Urysohn space is the set $\N$ equipped with the discrete metric. The $\{0,\frac 12,1\}$-Urysohn space is the \emph{Random graph} $\mathcal R$ equipped with the metric discussed in Example \ref{ex: graphs as metric spaces}. The $\Q\cap[0,1]$-Urysohn space is called the \emph{rational Urysohn sphere}.
\end{remark}

\subsection{Finitely supported extensions and amalgamation}

We now recall the construction of amalgamations of metric spaces and recast the notion of finitely supported extension in this context. The material in this section is standard. 

Let $(Y,d)$ be an $S$-metric space, suppose $X$ is a subset of $Y$ and let $y\in Y$. A subset $F\subseteq X$ is called a \textbf{support} for $y$ over $X$ if  for all $x\in X$, we have 
$$d(x,y)=\min\left(1,\min_{f\in F}\left(d(x,f)+d(f,y)\right)\right).$$
Observe that every subset containing $F$ is then also a support  of $y$. 

We say that a point $y\in Y$ is \textbf{finitely supported} over $X$ if it has a finite support. 
Note that every element $x$ of $X$ is finitely supported with support  $\{x\}$. 


We say that a metric space $(Y,d)$ is a \textbf{finitely supported extension} of a subset $X$ if all its points are finitely supported over $X$.
\begin{definition}
Suppose $(X_1,d_1)$ and $(X_2,d_2)$ are two $S$-metric spaces, that $A=X_1\cap X_2$ and $d_{1\restriction A}=d_{2\restriction A}$. Then we define the \textbf{metric amalgam} of $(X_1,d_1)$ and $(X_2,d_2)$ over $A$ as the set $X_1\cup X_2$ equipped with the metric $d$ which restricts to $d_1$ on $X_1$ and $d_2$ and $X_2$ and such that for all $x_1\in X_1$ and all $x_2\in X_2$ we have
$$d(x_1,x_2)=\min(1,\inf_{a\in A}(d_1(x_1,a)+d_2(a,x_2))).$$
We denote the metric amalgam by $(X_1,d_1)*_A(X_2,d_2)$. Note that we allow amalgamation over the empty set, in which case  $d(x_1,x_2)=1$ for all $x_1\in X_1$ and all $x_2\in X_2$. The fact that we have a well-behaved way of amalgamating $S$-metric spaces over the emptyset is a notable difference between the bounded case and the unbounded case. 
\end{definition}
\begin{remark}
If $X\subseteq (Y,d)$, a point $y\in Y$ is finitely supported over $X$ with support  $F\Subset X$ if and only if $X\cup \{y\}$ is the metric amalgam of $F\cup\{y\}$ and $X$ over $F$. 
\end{remark}
It is not clear a priori when the metric amalgam of two $S$-metric spaces is an $S$-metric space. However, it is  the case when $X_1$ and $X_2$ are finitely supported over $A$.

\begin{proposition}\label{prop: metric amalgam}
Let $(X_1,d_1)$ and $(X_2,d_2)$ be two $S$-metric spaces which are finitely supported extensions of a common metric subspace $A=X_1\cap X_2$. Then the metric amalgam $(X_1,d_1)*_A(X_2,d_2)$ is an $S$-metric space. 

Moreover if $F_1$ is a support  of $x_1\in X_1$ and $F_2$ is a support  of $x_2\in X_2$, and $\inf_{a\in A}d_1(x_1,a)+d_2(a,x_2)<1$, then this infimum is a minimum which is attained both on $F_1$ and $F_2$. 
\end{proposition}
\begin{proof}
Pick $x_1 \in X_1$, $x_2 \in X_2$. If $d(x_1,x_2)=1$ then this distance belongs to $S$, so we only have to deal with the case when

$$d(x_1,x_2)= \inf_{a\in A}d_1(x_1,a)+d_2(a,x_2)<1$$

By symmetry, we only need to check that if $F_1\Subset A$ is a support  of $x_1$, then  $\inf_{a\in A}d(x_1,a)+d_2(a,x_2)$ is attained on $F_1$. 

By definition, for each $a\in A$ we have $d(x_1,a)=\min_{f\in F_1}d(x_1,f)+d(f,a)$. So for each $a\in A$ there is $f\in F_1$ such that
$$d_1(x_1,a)+d_2(a,x_2)=d_1(x_1,f)+d_1(f,a)+d_2(a,x_2)$$
But $d_1$ and $d_2$ coincide on $A$ so $d_1(f,a)+d_2(a,x_2)=d_2(f,a)+d_2(a,x_2)\geq d_2(f,x_2)$. We conclude $d_1(x_1,a)+d_2(a,x_2)\geq d_1(x_1,f)+d_2(f,x_2)$, so our infimum is indeed a minimum attained on $F_1$.
\end{proof}

We will also make use of the following fact.

\begin{lemma}\label{lem: support  at distance 1 implies distance 1}
Suppose $(Y,d)$ is an $S$-metric space, $X\subseteq Y$ and $y_1$, $y_2\in Y$ are finitely supported over $X$, with supports $F_1,F_2$, and $X \cup \{y_1,y_2\}$ is the metric amalgam of $X\cup\{y_1\}$ and $X \cup \{y_2\}$ over $X$. 
Suppose moreover that $d(f_1,f_2)=1$ for all $f_1\in F_1$ and all $f_2\in F_2$. 
Then $d(y_1,y_2)=1$.
\end{lemma}
\begin{proof}
This follows from Proposition \ref{prop: metric amalgam}: assume $d(y_1,y_2)< 1$. Then there exists $f_1 \in F_1$ such that $d(y_1,y_2)=d(y_1,f_1)+ d(f_1,y_2)$, and by definition of a support there exists $f_2 \in F_2$ such that $d(f_1,y_2)=d(f_1,f_2)+d(f_2,y_2)$. Hence $d(f_1,f_2)<1$, contradicting our assumption.
\end{proof}

The construction of metric amalgams makes sense with an arbitrary number of factors: when $(X_i,d_i)_{i\in I}$ is a family of $S$-metric spaces which are finitely supported over $A$ and for all $i\neq j$ we have  $A=X_i\cap X_j$  and $d_{i\restriction A}=d_{j\restriction A}$ then we can form the metric amalgam of $(X_i,d_i)_{i\in I}$ over $A$ as the set $\bigcup_{i\in I} X_i$ equipped with the metric $d$ which coincides with $d_i$ when restricted to $X_i$, and such that for all $i\neq j$  and all $x_i\in X_i$ and $x_j\in X_j$ we have 
$$d(x_i,x_j)=\min(1,\inf_{a\in A}(d_i(x_i,a)+d_j(a,x_j))).$$
Observe that when restricted to $X_i\cup X_j$, the metric is the one of the metric amalgam of $X_i$ and $X_j$ over $A$. Proposition \ref{prop: metric amalgam} has then the following immediate corollary.

\begin{corollary}\label{cor: amalgam of finitely supported S metric is S metric}
Let $(X_i,d_i)_{i\in I}$ be a family of $S$-metric spaces  such that for all $i\neq j$ we have  $A=X_i\cap X_j$  and $d_{i\restriction A}=d_{j\restriction A}$. Suppose that each $X_i$ is finitely supported over $A$. Then the metric amalgam of $(X_i,d_i)_{i\in I}$ over $A$ is an $S$-metric space. \qed
\end{corollary}

\subsection{$S$-Urysohn spaces and one-point extensions} \label{sec: S urysohn and one point extensions}


Let us now see how to build $\mathbb U_S$, i.e. how to build a countable $S$-metric space with the extension property using a slight variation of  Kat\v{e}tov's ideas \cite{katetov}.

Given an  $S$-metric space $(X,d)$ and another $S$-metric space $(Y,d)$ containing $X$, we say that $Y$ is a \textbf{one-point extension} of $X$ if $\abs{Y\setminus X}=1$. A straightforward induction yields that $\mathbb U_S$ is characterized by the following version of the extension property: for every one-point extension $(Y,d)$ of every finite $S$-metric space $(X,d)$, any isometric embedding $\rho \colon X\to \US$ extends to an isometric embedding $\tilde \rho\colon Y\to \US$.

One-point extensions are completely determined by the distance to the added point. Such distance functions can be  characterized as  the functions $f\colon X\to S$ satisfying 
\begin{equation}\label{Katetov condition}
\vert f(x_1)-f(x_2)\vert\leq d(x_1,x_2)\leq f(x_1)+f(x_2)\quad\text{for all }x_1,x_2\in X.
\end{equation}
The functions $f:X\to S$ satisfying the above condition are thus also called one-point extensions of $X$, and we will often switch between the two point of views. 

\begin{definition}\label{def: space of f.s. one point extensions}
Whenever $(X,d)$ is an $S$-metric space, we denote by $E_S(X)$ the space of all \emph{finitely supported} one-point extensions of $X$. 
By definition these are the functions $f\colon X\to S$ satisfying \eqref{Katetov condition} such that there is $F\Subset X$ satisfying for all $x\in X$,
\[
f(x)=\min\left(1,\min_{y\in F} (f(y)+d(y,x))\right).
\]
Such a subset $F$ is called a \textit{support} for $f$. Supports are not unique: any finite set containing a support of $f$ is a support of $f$.
\end{definition}

Note that if $F\Subset X$, any one-point extension $f$ of $(F,d)$ extends to a finitely supported one-point extension $\tilde f$ of $X$ defined by 
$$\tilde f(x)=\min\left(1,\min_{y\in F}\left( f(y)+d(y,x)\right)\right).$$
This extension is called the \textbf{Kat\v{e}tov extension} of $f$, and may be viewed as the metric amalgam of $F\cup\{f\}$ and $X$ over $F$. 

A key idea of the \emph{Kat\v{e}tov construction} of the $S$-Urysohn space is that $X$ embeds isometrically into $E_S(X)$ as the space of trivial one-point extensions, i.e. extensions of the form $\hat x=d(x,\cdot)$. Moreover we have $\norm{f-\hat x}_\infty=f(x)$, so $E_S(X)$ contains as a metric subspace every finitely supported one-point extension of $X$, and hence every one-point extension of every finite subset of $X$. Also note that $E_S(X)$ is countable if $X$ was countable. 

However the metric induced by $\norm{\cdot}_\infty$ on $E_S(X)$ is not the right one for us because we want the one point extensions to be as far from each other as possible. The reason for this will become apparent when we construct for every countable group an action on the $S$-Urysohn space which is as free as possible.

So for $f,g\in E_S(X)$ we define a new metric still denoted by $d$, by setting 
\[
d(f,g):=\left\{\begin{array}{lcl}\min\left(1,\displaystyle\inf_{x\in X}\left(f(x)+g(x)\right)\right)&\text{if}&f\neq g,\\0&\text{if}&f=g.\end{array}\right.
\]
Observe that $d$ is the metric obtained by amalgamating all the finitely supported one-point extensions of $(X,d)$ over $X$, which as a set is still $E_S(X)$. In particular, $d$ is indeed a metric and takes values in $S$ by Corollary \ref{cor: amalgam of finitely supported S metric is S metric}. By construction $(X,d)$ is a metric subspace of $(E_S(X),d)$. 

One way of constructing the $S$-Urysohn space is now to start with an arbitrary countable metric space $(X,d)$, let $X_0=X$ and then define by induction $X_{n+1}=E_S(X_n)$. The metric space $\bigcup_{n\in\N} X_n$ then satisfies the one-point extension property by construction, and hence it is isometric to the $S$-Urysohn space $\mathbb U_S$. Moreover the construction is equivariant in the sense that if $\Gamma$ acts on $X$ by isometries, then the action naturally extends to $E_S(X)$ so that in the end we get a $\Gamma$-action by isometries on $\mathbb U_S$. In order to make the action as free as possible, we will however need to modify further this construction. Let us start by making clear what we mean by as free as possible.

\subsection{Freeness notions for actions on metric spaces}
Let $\Gamma$ be a countable group. If $(X,d)$ is a metric space, we write $\Gamma\curvearrowright (X,d)$ if $\Gamma$ acts on $X$ by isometries. A $\Gamma$-action by isometries on $X$ is thus a group homomorphism $\Gamma\to \Iso(X,d)$.

 We will need the following three notions for an action $\Gamma\curvearrowright (X,d)$ where $X$ is an $S$-metric space.


\begin{definition} An action $\Gamma\curvearrowright (X,d)$ is \textbf{mixing} if for every $x,y\in X$ and for all but finitely many $\gamma\in\Gamma$ we have $d(\gamma x,y)=1$.
\end{definition}

Observe that if $\Gamma\curvearrowright (X,d)$ is mixing, then for any finite subset $F\subseteq X$, for all but finitely many $\gamma\in\Gamma$ we have $d(\gamma x,y)=1$ for all $x,y\in F$.
Also note that the action of any finite group is mixing. Finally, if $\Gamma\curvearrowright (X,d)$ is mixing then the restriction of the action to every subgroup $\Lambda\leq\Gamma$ is mixing. 

%


\begin{definition}
An action $\Gamma\curvearrowright (X,d)$ is \textbf{strongly free} if, for all $\gamma\in\Gamma\setminus\{1\}$ and all $x\in X$, one has $d(\gamma x,x)=1$.
\end{definition}
Note that any strongly free action is obviously free. 
\begin{definition} 
Let $\Sigma<\Gamma$ be a subgroup. We say that $\Sigma$ is \textbf{highy core-free} for the action $\Gamma\curvearrowright (X,d)$ if for all 
$F\Subset X$, there exists $g\in\Gamma$ such that $d(gx,u)=1=d(\sigma gx,gy)$ for all $x,y\in F$, $u\in\Sigma F$ and $\sigma\in\Sigma\setminus\{1\}$. 
\end{definition}

\begin{remark}
The fact that $d(\sigma gx,gy)=1$ for all $x,y\in F,\sigma\in\Sigma\setminus\{1\}$ implies that the $\Sigma$-action on $\Sigma gF$ is conjugate to the $\Sigma$-action on $\Sigma\times F$ given by $\sigma(\sigma',x)=(\sigma\sigma',x)$ where we put on $\Sigma\times F$ the metric $\tilde d$ given by 
\[
\tilde d((\sigma,x),(\sigma',y))=
\left\{\begin{array}{cc}
d(x,y) & \text{if } \sigma=\sigma' \\
1 & \text{else.}
\end{array}\right.
\]
\end{remark}

Note that if $\Sigma$ is highly core-free for an action $\Gamma\curvearrowright (X,d)$ then any subgroup of $\Sigma$ is highly core-free for $\Gamma\curvearrowright (X,d)$. 

Our terminology extends that from \cite{zbMATH06503094}: it is easy to check that, if $d$ is the discrete metric then $\Sigma$ is highly core-free for $\Gamma\curvearrowright (X,d)$ in the sense above if and only if $\Sigma$ is strongly highly core-free for the action $\Gamma\curvearrowright X$ on the \textit{set} $X$ in the sense of \cite[Definition 1.7]{zbMATH06503094}. If moreover the action $\Gamma\curvearrowright X$ is free, then the previous conditions are equivalent to $\Sigma$ being highly-core-free in $\Gamma$ (see \cite[Definition 1.1 and Lemma 1.6]{zbMATH06503094}). As mentioned in the introduction, easy examples of highly core-free subgroups are the trivial subgroup of an infinite group and any finite subgroup of an i.c.c. group. 
A less immediate example is the fact that in the free group over $2$ generators 
$\mathbb F_2=\la a,b\ra$, the cyclic subgroup generated by the commutator $[a,b]$ is highly core-free. This fact  will allow us to show that surface groups arise as dense subgroups in the isometry group of any bounded $S$-Urysohn space, following the same strategy as in \cite[Ex. 5.1]{zbMATH06503094}.
We refer to \cite[Ex. 1.11]{zbMATH06503094} for proofs and for more examples of highly core-free subgroups.

Let us now see how these notions behave with respect to $E_S(X)$ (see Definition \ref{def: space of f.s. one point extensions}). 
Note that we have an injective group homomorphism $\Iso(X,d)\rightarrow\Iso(E_S(X),d)$ defined by $\varphi\mapsto(f\mapsto\varphi(f))$, where $\varphi(f)(x)=f(\varphi^{-1}(x))$. In particular, from any action $\Gamma\curvearrowright X$ we get an action $\Gamma\curvearrowright E_S(X)$ defined by $(\gamma f)(x)=f(\gamma^{-1} x)$ for $\gamma\in\Gamma$ and $f\in E_S(X)$.

\begin{proposition}\label{inductionstep}
Let $\Gamma\curvearrowright(X,d)$ be an action by isometries. The following holds.
\begin{enumerate}
\item\label{inductionstep1} If $\Gamma\curvearrowright (X,d)$ is faithful then $\Gamma\curvearrowright E_S(X)$ is faithful.
\item\label{inductionstep2} If $\Gamma\curvearrowright (X,d)$ is mixing then $\Gamma\curvearrowright E_S(X)$ is mixing.
\item\label{inductionstep3} If $\Sigma<\Gamma$ is highly core-free for $\Gamma\curvearrowright (X,d)$ then it is also highly core-free for $\Gamma\curvearrowright E_S(X)$.
\end{enumerate}
\end{proposition}

\begin{proof}

\eqref{inductionstep1} is obvious since we have a $\Gamma$-equivariant inclusion $X\subseteq E_S(X)$.

\eqref{inductionstep2}. Let $f,g\in E_S(X)$, let $F$ be a common finite support for $f$ and $g$. Since $\Gamma\curvearrowright X$ is mixing, for all but finitely many $\gamma\in \Gamma$ we have $d(\gamma x,y)=1$ for all $x,y\in F$. Note $\gamma F$ is a support  for  $\gamma f$. It now follows from Lemma \ref{lem: support  at distance 1 implies distance 1} that $d(\gamma f,g)=1$ for all but finitely many $\gamma\in \Gamma$. 

\eqref{inductionstep3}. Let $F\subseteq E_S(X)$ be a finite subset. For each $f\in F$ let $Y_f\Subset X$ be a support for $f$ and define $Y=\cup_{f\in F} Y_f\subseteq X$. If $\gamma\in\Gamma$ is such that $d(\gamma x,u)=1$ and $d(\sigma\gamma x,\gamma y)=1$ for all $x,y\in Y$, $u\in\Sigma Y$ and $\sigma\in\Sigma\setminus\{1\}$ then it follows from Lemma \ref{lem: support  at distance 1 implies distance 1} that $d(\gamma f,h)=1$ and $d(\sigma \gamma f,\gamma g)=1$ for all $f,g\in F$, $h\in\Sigma F$ and $\sigma\in\Sigma\setminus\{1\}$.
\end{proof}

Observe that we did not mention strong freeness, which indeed does not carry over to $E_S(X)$. For instance, if $X$ is finite and we let $f\in E_S(X)$ be defined by $f(x)=1$ for all $x\in X$, then $f$ is actually fixed by the $\Gamma$-action. One can tweak the above construction of $\U_S$ to get rid of this obstruction, as we will see in section \ref{sec: strongly free action on Urysohn}.

\subsection{Homogeneous actions on countable metric spaces}
When all the finite partial isometries of a metric space $(X,d)$ extend to isometries, we say that $X$ is \textbf{homogeneous}. By a back-and-forth argument,  $\U_S$ is a homogeneous metric space.  

The isometry group of the $S$-Urysohn space $\mathbb U_S$ is endowed with the topology of pointwise convergence for the discrete topology on $\mathbb U_S$. It then becomes a Polish group, and our main motivation is to understand which countable groups arise as dense subgroups of $\Iso(\mathbb U_S)$. Recall from the introduction that this is the same as asking which countable groups admit a \emph{faithful} action by isometries on $\mathbb U_S$ such that every finite partial isometry of $\mathbb U_S$ extends to an element of $\Gamma$. Such actions are called \textbf{homogeneous actions}, and we denote by $\HUS$ the class of countable groups which admit a faithful homogeneous action on $\US$.


One can use the methods of K. Tent and M. Ziegler \cite{MR3104993} (see also \cite{MR3022717} for considerations on the unbounded case) to show that the isometry group of any countable $S$-Uryshohn space is \emph{topologically simple} (every non-trivial normal subgroup is dense). This yields via the following proposition some restrictions on the class $\HUS$.
  
\begin{proposition}\label{prop: conditions dense subgroups of topo simple}
Let $G$ be a nonabelian infinite topologically simple topological group, and suppose $\Gamma$ is a dense subgroup. Then every nontrivial normal subgroup of $\Gamma$ is dense in $G$ and every finite index subgroup of $\Gamma$ is dense in $G$. Moreover,
$\Gamma$ is i.c.c. and non-solvable.
\end{proposition}

\begin{proof}

The first statement follows from the fact that for any normal subgroup $N\leq \Gamma$ its closure $\overline{N}$ is, by density of $\Gamma$, normal in $G$. So either $N=\{1\}$ or $N$ is dense in $G$. Moreover, any finite index subgroup $\Sigma\leq\Gamma$ has a finite index subgroup $N\leq\Sigma$ which is normal in $\Gamma$ and non-trivial since it has to be infinite. It follows that any finite index subgroup of $\Gamma$ is again dense in $G$. 

To prove the second statement, first observe that the center $Z(\Gamma)$ of a dense subgroup $\Gamma\leq G$ must be trivial. Indeed, $Z(\Gamma)$ is a normal subgroup of $\Gamma$ which is abelian, but $G$ is not, so the center of any  dense subgroup of $G$ must be trivial. 

Now suppose that $\gamma\in \Gamma$ has a finite conjugacy class and denote by $C(\gamma):=\{\gamma'\in \Gamma \colon \gamma\gamma'=\gamma'\gamma\}$ its centraliser. By definition $C(\gamma)$ has finite index in $\Gamma$ so by the first statement it is dense and thus has trivial center, from which we deduce that $\gamma$ is trivial since $\gamma$ belongs to the center of $C(\gamma)$. We conclude $\Gamma$ is i.c.c.

Finally, for a group $\Lambda$, let us denote by $D(\Lambda)$ the subgroup of $\Lambda$ generated by all commutators. Recall that $D(\Lambda)$ is normal in $\Lambda$. Now define $\Gamma_0:=\Gamma$ and, for $n\geq 0$, $\Gamma_{n+1}:=D(\Gamma_n)$. We show by induction that for all $n\geq 0$ the subgroup $\Gamma_n$ is dense in $G$, which implies that $\Gamma$ is not solvable. The case $n=0$ is true by assumption. Suppose $\Gamma_n$ is dense in $G$, then $\Gamma_{n+1}$ is either trivial or dense in $G$ since it is normal in $\Gamma_n$. However, $\Gamma_{n+1}$ is not trivial since $\Gamma_n$, being dense in $G$, can not be abelian.
\end{proof}

\begin{corollary}\label{cor: restrictions HUS}
If $\Gamma\in\HUS$ for some distance set $S$  then $\Gamma$ is i.c.c. and non solvable. 
Furthermore, every finite index subgroup of $\Gamma$ belongs to $\HUS$. 
\qed
\end{corollary}

When $S=\{0,1\}$, recall that $\HUS$ is the class of highly transitive groups. Hull and Osin have furthermore shown the following dichotomy: a highly transitive group either contain a copy of the infinite alternating group, or it is \emph{mixed identity free}, which has stronger implications than those of the previous corollary (see \cite[sec. 5]{zbMATH06673802}).

\begin{remark}
We do not know whether in Proposition \ref{prop: conditions dense subgroups of topo simple}, for say a Polish group $G$, one can remove the hypothesis that $G$ is nonabelian. Indeed, there may exist a topologically simple  infinite Polish abelian group, in other words there may exist an infinite abelian Polish  group with no nontrivial proper closed subgroup (see the end of section 2 in \cite{Hooper_1976}, where it is also proved that there is an infinite Polish abelian group all whose locally compact subgroups are trivial). 
\end{remark}


\section{Strongly free actions on the $S$-Urysohn space}\label{sec: strongly free action on Urysohn}

\subsection{Equivariant extension of partial isometries}

\begin{definition}
A $\Gamma$-action by isometries on an $S$-metric space $X$ has the \textbf{extension property} if every  finitely supported one-point extension over a finite union of $\Gamma$-orbits is realized in $X$. 
\end{definition}
Let us make a few remarks about this definition. First, it is clear that the following assertions are equivalent.
\begin{itemize}
\item $X$ has the extension property.
\item The trivial action on $X$ of the trivial group has the extension property.
\item The action of any finite group on $X$ has the extension property.
\end{itemize}

Also, if $\Gamma\curvearrowright X$ has the extension property, then for every $\Lambda\leq\Gamma$, the $\Lambda$-action on $X$ also has the extension property. 
Indeed if $f$  is a finitely supported one-point extension over $\Lambda F$ where $F\Subset X$, then the Kat\v{e}tov extension of $f$ to $\Gamma F$ is also finitely supported (with same support), and hence realized in $X$. In particular, if there exists a group $\Gamma$ acting on $X$ such that the action $\Gamma\curvearrowright X$ has the extension property then $X$ has the extension property (since the action of the trivial subgroup of $\Gamma$ has the extension property).

Finally, note that any action with the extension property has infinitely many orbits at distance $1$ from each other because the one point extension $f(x)=1$ is finitely supported.

%
%

\begin{definition}
Let $\Sigma$ be a group and for $k=1,2$, let $\pi_k\colon \Sigma\curvearrowright(X_k,d_k)$. A $(\pi_1,\pi_2)$-partial isometry is an isometry between $\pi_1(\Sigma)F_1$ and $\pi_2(\Sigma)F_2$ where $F_1\Subset X_1$ and $F_2\Subset X_2$ which is $(\pi_1,\pi_2)$-equivariant. A $(\pi_1,\pi_2)$-isometry is a $(\pi_1,\pi_2)$-equivariant bijective isometry from $X_1$ to $X_2$.
\end{definition}

\begin{lemma}\label{lem: mixing brings finitely supported extensions}
Suppose $\Gamma\curvearrowright (X,d)$ is a mixing action and $F$ is a finite subset of $X$. Then for every $x\in X$, the function $d(x,\cdot)$ is a finitely supported one-point extension of $\Gamma F$. 
\end{lemma}
\begin{proof}
Since the action is mixing, there are only finitely many points $y\in\Gamma F$ such that $d(x,y)<1$. These points form a support for $d(x,\cdot)$.
\end{proof}

The next proposition is the key building block of most of our constructions.

\begin{proposition}\label{Extension}
For $k=1,2$, let $\pi_k\colon \Sigma\curvearrowright \U_S$ be actions of a countable group $\Sigma$ which are strongly free, mixing and have the extension property.

Then, every $(\pi_1,\pi_2)$-partial isometry $\varphi$ extends to a $(\pi_1,\pi_2)$-isometry of $\U_S$. 
\end{proposition}

\begin{proof}
A straightforward back-and-forth argument yields that we only need to show that for every $(\pi_1,\pi_2)$-partial isometry $\varphi$  and every
 $x\in \U_S\setminus\dom\varphi$, there is a $(\pi_1,\pi_2)$-partial isometry $\tilde \varphi$ whose domain contains $x$ and which extends $\varphi$. 
Let $f$ be the one-point extension of $\rng \varphi$ defined by $f(y)=d(\varphi^{-1}(y),x)$ for all $y\in \rng\varphi$.  

Because $\pi_1$ is mixing, Lemma \ref{lem: mixing brings finitely supported extensions} yields $f$ is finitely supported. 
By the extension property, there is $z\in \mathbb U_S$ such that for all $y\in\rng\varphi$ we have 
$d(\varphi\inv(y),x)=d(y,z)$.

Observe that the extension of $\varphi$ obtained by sending $x$ to $z$ is by construction a partial isometry. By freeness, we may extend $\varphi$ further to a $(\pi_1,\pi_2)$-equivariant partial isometry $\tilde \varphi\colon \dom\varphi\sqcup\pi_1(\Sigma)x\to\rng\varphi\sqcup\pi_2(\Sigma)z$ by letting $\tilde\varphi(\pi_1(\sigma)x)=\pi_2(\sigma) z$.
 
Let us check that $\tilde \varphi$ is indeed a partial isometry. First, by strong freeness its restriction to $\pi_1(\Sigma)x$ is isometric. Since $\varphi$ was also isometric, we only need to check that if $x_1\in \dom\varphi$ and $x_2\in \pi_1(\Sigma) x$ then $d(\varphi(x_1),\tilde\varphi(x_2))=d(x_1,x_2)$. 

Let $\sigma\in\Sigma$ be such that $x_2=\pi_1(\sigma)x$. Then 
\begin{align*}
d(x_1,x_2)&=d(\pi_1(\sigma\inv)x_1, x)\\
&=d(\varphi(\pi_1(\sigma\inv)x_1),z)\\
&=d(\pi_2(\sigma\inv)\varphi(x_1),z)\\
&=d(\varphi(x_1),\pi_2(\sigma)z)\\
&=d(\varphi(x_1),\tilde\varphi(x_2)).
\end{align*}
So $\tilde \varphi$ is indeed a $(\pi_1,\pi_2)$-partial isometry.
\end{proof}

\subsection{Extensions with parameters} \label{sec: ext with params}

Let $(Y,d_Y)$ be a $S$-metric space, suppose $X\subseteq Y$ and $A$ is a set. Then we can consider the space $Y\times A$ equipped with the pseudometric
$$d((y,a),(y',a'))=\left\{
\begin{array}{cl}
\min\left(1,\inf_{x\in X} \left(d_Y(y,x)+d_Y(x,y')\right)\right)&\text{ if }a'\neq a, \\
d_Y(y,y') &\text{ if }a'=a.
\end{array}
\right.
$$
Then the metric space obtained by identifying elements at distance $0$ is denoted by $Y\times_X A$, and the induced metric still by $d$. 
Observe that for $x,x'\in X$ and $a,a'\in A$,  
$$d((x,a),(x',a'))=\inf_{x''\in X}\left(d_Y(x,x'')+d_Y(x'',x')\right)=d_Y(x,x')$$
so $Y\times_XA$ is the metric amalgam of $\abs A$ copies of $Y$ over $X$ and thus $Y\times_XA$ is a $S$-metric space as soon as $Y$ is a finitely supported extension of $X$.

We have the following  inequality: for all $(y,a)$ and $(y',a')$ in $Y\times_X A$,
\begin{equation}\label{eqn: distances are dilated}
d((y,a),(y',a'))\geq d_Y(y,y')
\end{equation}
because by the triangle inequality for all $x$ we have $d_Y(y,x)+d_Y(x,y')\geq d_Y(y,y')$. 

We also have for each $a\in A$ an isometry $i_a\colon Y\to Y\times_XA$ which takes $y$ to $(y,a)$. In particular $Y$ embeds in  $Y\times_XA$ isometrically.

For all $a\neq a'$ and all $x\in X$ we have $i_a(x)=i_{a'}(x)$ and we thus let $i\colon X\to Y\times_XA$ be the common map.

\begin{remark}Every isometry $g$ of $Y$ which preserves $X$ naturally extends to $Y\times_X A$ by letting $g(y,a)=(gy,a)$. The construction we need is different, however.
\end{remark}
\begin{proposition}
Suppose that we have a group $\Gamma$ acting on $Y$ by isometries which preserve the set $X$, and suppose moreover that $\Gamma$ acts on $A$. Then the diagonal action on $Y\times A$ defined by $\gamma(y,a)=(\gamma y,\gamma a)$ induces an isometric action on $(Y\times_X A,d)$ which extends the $\Gamma$-action on $X$.
\end{proposition}
\begin{proof}
Let us check that every $\gamma\in \Gamma$ defines an isometry for the pseudometric $d$. Let $a,a'\in A$ and $y,y'\in Y$.
If $a=a'$, then $d(\gamma(y,a),\gamma(y',a'))=d_Y(\gamma y,\gamma y')=d((y,a),(y',a'))$. If $a\neq a'$ then $\gamma a\neq \gamma a'$ and thus
\begin{align*}
d(\gamma (y,a),\gamma (y',a'))&=\min\left(1,\left(\inf_{x\in X} d_Y(\gamma y,x)+d_Y(x,\gamma y')\right)\right)\\
&=\min\left(1,\left(\inf_{x\in X} d_Y(\gamma y,\gamma x)+d_Y(\gamma x,\gamma y')\right)\right)
\end{align*}
because $X=\gamma X$. Now using the fact that $\gamma$ is an isometry, we conclude
\begin{align*}
d(\gamma (y,a),\gamma (y',a'))
&=\min\left(1,\left(\inf_{x\in X} d_Y(y,x)+d_Y(x,y')\right)\right)\\
&=d((y,a),(y',a')).
\end{align*}

Finally observe that the embedding $i\colon X\to Y\times_X A$ is $\Gamma$-equivariant and thus conjugates the $\Gamma$-action on $X$ to the $\Gamma$-action on $i(X)$.
\end{proof}

We call the above action the \textbf{diagonal action} of $\Gamma$ on $Y\times_X A$. 
\begin{proposition}
Let $\Gamma$ be a group acting by isometries on $(Y,d_Y)$, let $X\subseteq Y$ be $\Gamma$-invariant subset and let $A$ be a set also acted upon by $\Gamma$. The following hold.
\begin{enumerate}
\item If the $\Gamma$-action on  $X$ is faithful, then the $\Gamma$-action on $Y\times_X A$ is faithful.
\item If the $\Gamma$-action on $Y$ is mixing then the $\Gamma$-action on $Y\times_X A$ is mixing.
\item If $\Sigma\leq\Lambda\leq\Gamma$ and $\Sigma$ is highly core-free for the $\Lambda$-action on $Y$ then $\Sigma$ is also highly core-free for the $\Lambda$-action on $Y\times_X A$.
\item If the $\Gamma$-action on $X$ is strongly free and the $\Gamma$-action on $A$ is free, then the $\Gamma$-action on $Y\times_X A$ is strongly free.
\end{enumerate}
\end{proposition}
\begin{proof}
(1) It is obvious since the isometry $ i\colon  X\rightarrow Y\times_X A$ is $\Gamma$-equivariant.

(2) Suppose the $\Gamma$-action on $Y$ is mixing. Let $(y_1,a_1), (y_2,a_2)\in Y\times_X A$.  By assumption for all but finitely many $\gamma\in \Gamma$ we have $d_Y(\gamma y_1,y_2)=1$, so by inequality \eqref{eqn: distances are dilated} for all but finitely $\gamma\in\Gamma$ we have $d(\gamma(y_1,a_1),(y_2,a_2))=1$. 

(3) Let $\Lambda\leq \Gamma$, let $\Sigma\leq \Lambda$. Suppose $\Sigma$ is highly core-free for the $\Lambda$-action on $Y$. Let $F\Subset Y\times_X A$  and write $F= F_1\times\{a_1\}\sqcup \ldots \sqcup F_n\times\{a_n\}$. Let $\tilde F=\bigcup_{i=1}^n F_i$, then there is $\lambda\in\Lambda$ such that $d_Y(\lambda y,u)=1=d_Y(\sigma \lambda y,\lambda y')$ for all $y,y'\in F$, $u\in\Sigma F$ and all $\sigma\in\Sigma\setminus\{1\}$. As before by inequation \eqref{eqn: distances are dilated} we are done.

(4) Suppose the $\Gamma$-action on $X$ is strongly free and the $\Gamma$-action on $A$ is free. Let $(y,a)\in Y\times_XA$ and let $\gamma\in \Gamma \setminus\{1\}$. Then $\gamma a\neq a$ so 
\begin{align*}
d((y,a),\gamma(y,a))&=\inf_{x\in X} \left(d_Y(y,x)+d_Y(x,\gamma y)\right)\\
&=\inf_{x\in X} \left(d_Y(y,x)+d_Y(y,\gamma^{-1} x)\right)\\&\geq\inf_{x\in X} d_Y(x,\gamma^{-1} x).
\end{align*}
By strong freeness for all $x\in X$ we have $d_Y(x,\gamma^{-1} x)=1$. So we have $d((y,a),\gamma(y,a))=1$ and the $\Gamma$-action on $Y\times_X A$ is thus strongly free.
\end{proof}

Suppose now $\Gamma$ is a countable group acting on a $S$-metric space $(X,d)$ by isometries. Then the $\Gamma$-action extends to $E_S(X)$, and we define the $S$-metric space $E_S^\Gamma(X)=E_S(X)\times_X\Gamma$. Then the $\Gamma$-action by left translation on itself provides us a $\Gamma$-action on $E_S^\Gamma(X)$ which extends the $\Gamma$-action on $X$. We have the following facts.

\begin{proposition}\label{relativeinductionstep}
The following hold.
\begin{enumerate}
\item If the $\Gamma$-action on $X$ is faithful so is the $\Gamma$-action on $E_S^\Gamma(X)$.
\item If the $\Gamma$-action on $X$ is mixing then so is the $\Gamma$-action on $E_S^\Gamma(X)$.
\item If  $\Sigma\leq\Lambda\leq\Gamma$ and $\Sigma$ is highly core-free for the $\Lambda$-action on $X$ then $\Sigma$ is also highly core-free for the $\Lambda$-action on $E_S^\Gamma(X)$.
\item If the $\Gamma$-action on $X$ is strongly free then the $\Gamma$-action on $E_S^\Gamma(X)$ is strongly free.
\end{enumerate}
\end{proposition}
\begin{proof}
This is a straightforward application of the previous proposition along with Proposition \ref{inductionstep}.\end{proof}

\subsection{Induced actions on the $S$-Urysohn space}
Let $(X,d)$ be an $S$-metric space on which $\Gamma$ acts by isometries.
Define the sequence of $S$-metric spaces $X_0^\Gamma=X$ and, for $n\geq 0$, $X_{n+1}^\Gamma=E_S^\Gamma(X^\Gamma_n)$. Consider the inductive limit $\displaystyle X_\infty^\Gamma=\lim_{\rightarrow} X_n^\Gamma$ along with its natural $\Gamma$-action.

\begin{proposition}\label{RelativeInducedAction}
With the notations above, the following holds.
\begin{enumerate}
\item\label{induce1} $X^\Gamma_\infty\simeq\U_S$.
\item\label{induce2} If the $\Gamma$-action on $X$ is mixing then so is the $\Gamma$-action on $X^\Gamma_\infty$.
\item\label{induce3} If $\Sigma\leq\Lambda\leq \Gamma$ is highly core-free for the $\Lambda$-action on $X$ then it  is also highly core-free for the $\Lambda$-action on $X^\Gamma_\infty$.
\item\label{induce4} If the $\Gamma$-action on $X$ is strongly free then so is the $\Gamma$-action on $X^\Gamma_\infty$.
\item\label{induce5} The $\Gamma$-action on $X^\Gamma_\infty$ has the extension property. 
\end{enumerate}
\end{proposition}

\begin{proof}
Assertions \eqref{induce2} to \eqref{induce4} follow from Propositions \ref{relativeinductionstep} and the fact that these conditions hold on $X_\infty^\Gamma$ if and only if they hold on $X_n^\Gamma$ for each $n\in\N$.

Let us prove \eqref{induce5}. Let $F$ be a finite subset of $X^\Gamma_\infty$ and let $f$ be a finitely supported one-point extension of $\Gamma F$. Pick $N$ large enough that $F\subseteq X_N^\Gamma$. Since $X_N^\Gamma$ is $\Gamma$-invariant we also have  $\Gamma F\subseteq X_N^\Gamma$. Let $\tilde f$ be the Kat\v{e}tov extension of $f$ to $X_N^\Gamma$, observe that $\tilde f$ is still finitely supported with the same support as $f$ and thus $\tilde f\in E_S(X_N^\Gamma$). For all $x\in X_N^\Gamma$ we have $d(x,\tilde f)=\tilde f(x)$ where $d$ is the metric on $E_S(X_N^\Gamma)$, in particular $d(x,\tilde f)=f(x)$ for all $x\in \Gamma F$. Since $E_S(X_N^\Gamma)$ embeds isometrically in $X^\Gamma_\infty$ via a $\Gamma$-equivariant map which fixes $X_N^ \Gamma$, the proof of $(5)$ is complete. It now follows from $(5)$ that $X^\Gamma_\infty$ itself has the extension property. Hence, $X^\Gamma_\infty\simeq\U_S$.\end{proof}

\begin{corollary}
Every finite group admits a unique strongly free action on $\mathbb U_S$ up to conjugacy.
\end{corollary}
\begin{proof}
Let $\Gamma$ be a finite group, consider the left-action of $\Gamma$ on itself as an action by isometries where we put on $\Gamma$ the discrete metric. This action is clearly strongly free. By Proposition \ref{RelativeInducedAction} we can extend this action to a strongly free action of $\Gamma$ on $(\mathbb U_S,d)$. 
Now observe that any $\Gamma$-action on $\mathbb U_S$ is automatically mixing and has the extension property. Applying Proposition \ref{Extension} to the empty map, we see that our strongly free action is unique up to conjugacy.  
\end{proof}

\begin{corollary}\label{ActionU}
Any infinite countable group $\Gamma$ admits a unique (up to conjugacy) strongly free mixing action with the extension property on $\U_S$. This action has the property that, for any subgroups $\Sigma\leq\Lambda\leq\Gamma$ such that $\Sigma\leq \Lambda$ is a highly core-free subgroup,  $\Sigma$ is highly core-free for $\Lambda\curvearrowright\U_S$.
\end{corollary}

\begin{proof}
Uniqueness follows from Proposition \ref{Extension} applied to the empty map. To prove existence of such an action, and that it has the property mentioned above,
start from $\Gamma$ acting on $X=\Gamma$ by left translation where we equip $X$ with the discrete metric $d(x,y)=1$ for $x\neq y$ and $d(x,x)=0$. It is clear that $\Gamma\curvearrowright X$ is strongly free and mixing. Consider $\Gamma\curvearrowright X^\Gamma_\infty$, then by the previous proposition this action is a strongly free mixing action on the $S$-Urysohn space with the extension property. Now fix subgroups $\Sigma\leq\Lambda\leq\Gamma$ and suppose that $\Sigma\leq\Lambda$ is a highly core-free subgroup. It is clear that $\Sigma$ is highly core-free for $\Lambda\curvearrowright X$, since $d$ is the discrete metric and $\Lambda\curvearrowright X$ is free.  It follows from item (3) of the above proposition that $\Sigma$ is highly core-free for $\Lambda\curvearrowright X^\Gamma_\infty$.
\end{proof}


\section{Actions of amalgamated free products on bounded Urysohn spaces}
Let $\Gamma_1,\Gamma_2$ be two countable groups with a common subgroup $\Sigma$ and define $\Gamma=\Gamma_1\underset{\Sigma}{*}\Gamma_2$. Suppose that we have a faithful action $\Gamma\curvearrowright\U_S$ and view $\Gamma$ as a subgroup of $\Iso(\U_S)$.

Let $Z:=\{\alpha\in\autUQone \colon  \forall\sigma\in\Sigma \  \alpha\sigma=\sigma\alpha \}$. Note that $Z$ is a closed subgroup of $\autUQone$, hence a Polish group. Moreover, for all $\alpha\in Z$, there exists a unique group homomorphism $\pi_{\alpha} \colon  \Gamma \to \autUQone$ such that
$$\pi_{\alpha}(g)=\left\{\begin{array}{lcl}g&\text{if}&g\in\Gamma_1,\\
\alpha^{-1}g\alpha&\text{if}&g\in\Gamma_2.\end{array}\right.$$

 When $\Sigma$ is trivial, we have $Z=\autUQone$. In this section we prove the following result.

\begin{theorem}\label{ThmMain} If $\Gamma\curvearrowright\UQone$ is free with the extension property, $\Sigma\curvearrowright\UQone$ is strongly free, mixing and $\Sigma$ is highly core-free for $\Gamma_1,\Gamma_2\curvearrowright\U_S$ then the set:
$$O=\{\alpha\in Z\colon \pi_{\alpha}\text{ is homogeneous and faithful}\}$$
is dense $G_\delta$ in $Z$.
\end{theorem}

\noindent 
\begin{corollary}\label{cor: amalgamated examples}
Assume $\Gamma_1,\Gamma_2$ are two countable groups with a common subgroup $\Sigma$, highly core-free in both $\Gamma_1,\Gamma_2$. Then $\Gamma_1\underset{\Sigma}{*}\Gamma_2$ belongs to $\HUS$. In particular, the following results hold.
\begin{itemize}
\item For any infinite countable groups $\Gamma_1,\Gamma_2$ we have $\Gamma_1*\Gamma_2\in\HUS$.
\item For any i.c.c. groups $\Gamma_1,\Gamma_2$ with a common finite subgroup $\Sigma$ we have $\Gamma_1\underset{\Sigma}{*}\Gamma_2\in\HUS$.
\item If $\Sigma_g$ is a closed, orientable surface of genus $g>1$, then $\pi_1(\Sigma_g)\in\HUS$.
\end{itemize}
\end{corollary}
\begin{proof}
The first two results are a direct consequence of Theorem \ref{ThmMain} Corollary \ref{ActionU} and the examples of highly core-free subgroups given in \cite[Lem. 1.10]{zbMATH06503094}. Moreover, since in the free group over two generators $\mathbb F_2=\la a,b\ra$, the cyclic group generated by $[a,b]$ is highly core-free \cite[Ex. 1.11]{zbMATH06503094}, we also get that $\pi_1(\Sigma_2)$ belongs to $\HUS$. By covering theory, the latter admits every $\pi_1(\Sigma_g)$ for $g\geq 2$ as a finite index subgroup, so by Corollary \ref{cor: restrictions HUS} all such groups belong to $\HUS$.
\end{proof}

To prove the theorem, we establish two lemmas; the theorem follows by applying the Baire category theorem.

\begin{lemma}\label{LemHomogeneous}
Assume $\Sigma\curvearrowright\UQone$ is strongly free, mixing, has the extension property and $\Sigma$ is highly core-free for $\Gamma_1,\Gamma_2\curvearrowright\UQone$. Then the set $U=\{\alpha\in Z\colon \pi_{\alpha}\text{ is homogeneous}\}$ is dense $G_\delta$ in $Z$.
\end{lemma}

\begin{proof} For every finite partial isometry $\varphi$, consider the following open set
\[
U_\varphi:=\{\alpha\in Z\colon \exists g\in\Gamma\text{ such that }\pi_{\alpha}(g)_{\restriction \dom(\varphi)}=\varphi\}
.\]
Then $U=\bigcap_{\varphi\in P_f(\U_S)}U_\varphi$ so by the Baire category theorem it suffices to show that $U_\varphi$ is dense for every finite partial isometry $\varphi$. Let $\varphi\in P_f(\U_S)$, $\alpha\in Z$ and $F\Subset\U_S$. 
We need to prove that there are $\beta\in Z$ and $g\in\Gamma$ such that $\beta_{\restriction_F}=\alpha_{\restriction F}$ and ${\pi_\beta(g)}_{\restriction \dom(\varphi)}=\varphi$. 

Since $\Sigma$ is highly core-free for $\Gamma_1\curvearrowright\UQone$, there is $g_1\in\Gamma_1$ such that 
$$\forall u \in \Sigma F \ \forall x,x' \in \dom\varphi \ \forall \sigma\in\Sigma\setminus\{1\} \quad d(g_1x,u)=1=d(\sigma g_1x, g_1x'). $$
Define $F'= F \cup g_1\dom(\varphi)$. There exists $h_1\in\Gamma_1$ such that 
$$\forall u\in \Sigma F' \ \forall y,y'\in \rng(\varphi) \ \forall \sigma\in\Sigma \setminus \{1\} \quad d(h_1y,u)=1=d(\sigma h_1y,h_1y'). $$ 

Since $\Sigma$ is highly core-free for $\Gamma_2\curvearrowright\UQone$, there is $g_2\in\Gamma_2$ such that
$$ \forall u,v \in \alpha F' \ \forall w \in \Sigma \alpha F' \ \forall \sigma\in\Sigma \setminus \{1\} \quad d(\sigma g_2u, g_2v)=1=d(g_2u,w).$$
Define
$ A=\Sigma F' \sqcup \Sigma h_1\rng(\varphi)$ and $B=\alpha(\Sigma F')  \sqcup \Sigma g_2\alpha(g_1\dom(\varphi))$.

Note that $\Sigma A= A$ and, since $\alpha\in Z$, $\Sigma B=B$. 
Since the $\Sigma$-action is free, we can define a $\Sigma$-equivariant map $\beta_0\colon A\rightarrow B$ by setting
\begin{itemize}
\item $\forall u\in\Sigma F'  \quad \beta_0(u)=\alpha(u)$. 
\item $\forall x \in \dom(\varphi) \ \forall \sigma \in \Sigma \quad \beta_0(\sigma h_1\varphi(x))=\sigma g_2\alpha(g_1x)$. 
\end{itemize}
By the careful choices of the elements $g_1,g_2,h_1$ the map $\beta_0$ is a $\Sigma$-equivariant partial isometry. Since the $\Sigma$ action is strongly free and mixing with the extension property, we may apply Proposition \ref{Extension} to get an extension $\beta\in Z$ of $\beta_0$. Note that $\beta_{\restriction F}=\alpha_{\restriction F}$. Letting $g=h_1^{-1}g_2g_1\in \Gamma$ we have, for all $x\in \dom(\varphi)$, 
\[
\pi_\beta(g)x = h_1^{-1}\beta^{-1}g_2\beta g_1x
		    = h_1^{-1}\beta^{-1}g_2\alpha(g_1x)
		    = h_1^{-1}h_1 \varphi(x)
		    =\varphi(x)
\]
as wanted.
\end{proof}

\begin{lemma}\label{LemFaithful}
If $\Gamma\curvearrowright\UQone$ is free with the extension property, $\Sigma\curvearrowright\mathbb{U}_S$ is strongly free and mixing then the set $V=\{\alpha\in Z\colon \pi_{\alpha}\text{ is faithful}\}$ is a dense $G_\delta$ in $Z$.
\end{lemma}

\begin{proof}
We can write $V$ as a countable intersection of open sets $V=\cap_{g\in\Gamma\setminus \{1\}}V_g$, where $V_g=\{\alpha\in Z\colon \pi_{\alpha}(g)\neq\id\}$, so by the Baire category theorem it suffices to show that $V_g$ is dense for all $g\in\Gamma\setminus \{1\}$. Since the action $\Gamma\curvearrowright\UQone$ is faithful we have $V_g=Z$ for all $g\in (\Gamma_1\cup\Gamma_2)\setminus\{1\}$ and it suffices to show that $V_g$ is dense for all $g$ reduced of length $\geq 2$.

Let $g=g_{i_n}\dots g_{i_1}$, where $n\geq 2$ and $g_{i_k}\in\Gamma_{i_k} \setminus \Sigma$. Fix $\alpha\in Z$ and $F\Subset\UQone$ and define $F'=F\cup\alpha(F)\Subset\UQone$. Since $\Gamma\curvearrowright\UQone$ has the extension property, there exists $x\in\UQone$ such that $d(x,u)=1$ for all $u\in\Gamma F'$. Define: 
\[
\begin{array}{ccccccccc}
A&:=&\Sigma F&\sqcup&(\bigsqcup _{l=1}^n\Sigma g_{i_l}\dots g_{i_1}x)&\sqcup&\Sigma x&\text{ and }\\B&:=&\alpha(\Sigma F)&\sqcup&(\bigsqcup _{l=1}^n\Sigma g_{i_l}\dots g_{i_1}x)&\sqcup&\Sigma x&
\end{array}
\]
Note that $\Sigma A=A$ and, since $\alpha\in Z$, $\Sigma B=B$. Define the bijection $\gamma_0 \colon  A\rightarrow B$ by ${\gamma_0}_{\restriction \Sigma F}=\alpha_{\restriction \Sigma F}$ and ${\gamma_0}_{\restriction A\setminus\Sigma F}=\id$. Note that by definition of $x$, any element of $A\setminus(\Sigma F)$ (resp. $B\setminus\alpha(\Sigma F)$) is at distance $1$ of any element of $\Sigma F$ (resp. $\alpha(\Sigma F)$). It follows that $\gamma_0$ is an isometry and, since $\alpha\in Z$, we have $\gamma_0\sigma=\sigma\gamma_0$ for all $\sigma\in\Sigma$. 

Our assumptions on the action of $\Sigma$ enable us to apply Proposition \ref{Extension}, and find an extension $\gamma\in Z$ of $\gamma_0$. Then $\gamma_{\restriction F}=\alpha_{\restriction F}$ and $\pi_\gamma(g_{i_n}\dots g_{i_1})x=g_{i_n}\dots g_{i_1}x\neq x$, since $g=g_{i_n}\dots g_{i_1}\neq 1$ and the $\Gamma$-action is free. Hence $\gamma\in V_g$.
\end{proof}

When one of the factors is finite, say $\Gamma_2$, then, as explained in \cite{zbMATH06503094}, $\Gamma_2$ has no highly core-free subgroups. However, we have a similar result in that case.

\begin{theorem}\label{ThmOneFiniteFactor}
Suppose that $\Gamma_2$ is finite and $[\Gamma_2:\Sigma]\geq 2$. If $\Gamma\curvearrowright\UQone$ is strongly free with the extension property and $\Sigma$ is highly core-free for $\Gamma_1\curvearrowright\UQone$ then the set $$O:=\{\alpha\in Z\colon \pi_{\alpha}\text{ is homogeneous and faithful}\}$$
is a dense $G_\delta$ in $Z$.
\end{theorem}

As before, the following corollary follows from Theorem \ref{ThmOneFiniteFactor}, Corollary \ref{ActionU}, as well as the examples of highly core-free subgroups given in \cite{zbMATH06503094}.
\begin{corollary}
For every infinite countable group $\Gamma_1$, every finite group $\Gamma_2$ with common subgroup $\Sigma<\Gamma_1,\Gamma_2$ such that $\Sigma$ is highly core-free in $\Gamma_1$ and $[\Gamma_2:\Sigma]\geq 2$, we have $\Gamma_1\underset{\Sigma}{*}\Gamma_2\in\mathcal{H}_{\UQone}$. In particular, the following facts hold.
\begin{itemize}
\item If $\Gamma_1$ is countably infinite and $\Gamma_2$ is a finite non-trivial then $\Gamma_1*\Gamma_2\in\mathcal{H}_{\UQone}$.
\item If $\Gamma_1$ is i.c.c., $\Gamma_2$ is finite and $[\Gamma_2:\Sigma]\geq 2$ then $\Gamma_1\underset{\Sigma}{*}\Gamma_2\in\mathcal{H}_{\UQone}$.
\end{itemize}
\end{corollary}

In order to prove Theorem \ref{ThmOneFiniteFactor}, we prove the analogue of Lemma \ref{LemHomogeneous}. The theorem then follows as before from the combination of Lemmas \ref{LemHomogeneousOneFiniteFactor} and \ref{LemFaithful}.

\begin{lemma}\label{LemHomogeneousOneFiniteFactor}
Assume $\Gamma \curvearrowright\UQone$ is strongly free with the extension property, $\Sigma$ is highly core-free for $\Gamma_1\curvearrowright\UQone$, $\Gamma_2$ is finite and  $ [\Gamma_2:\Sigma ] \geq 2$. Then the set $U:=\{\alpha\in Z\colon \pi_{\alpha}\text{ is homogeneous}\}$ is dense $G_\delta$ in $Z$.
\end{lemma}

\begin{proof}
We again have to prove that, given $\varphi\in P_f(\UQone)$, $\alpha\in Z$ and $F\Subset \UQone$, we can find $\gamma\in Z$ and $g\in\Gamma$ such that $\gamma_{\restriction F}=\alpha_{\restriction F}$ and $\pi_\gamma(g)_{\restriction \dom(\varphi)}=\varphi$. 


Since $\Sigma$ is highly core-free for $\Gamma_1\curvearrowright\UQone$ and $\Gamma_2$ is finite, we may find $g_1,h_1\in\Gamma_1$ such that
\begin{itemize}
\item $\forall u\in \Gamma_2 F \ \forall x \in \dom(\varphi) \quad d(g_1x,u)=1$.
\item $\forall \sigma \in \Sigma  \ \setminus \{1\} \ \forall x,y \in \dom(\varphi) \quad d(\sigma g_1x,g_1y)=1$.
\item  $\forall u \in \Gamma_2 F \ \sqcup \Gamma_2 g_1 \dom(\varphi) \ \forall  y \in \rng(\varphi) \quad d(h_1y,u)=1$
\item $\forall \sigma \in \Sigma \setminus \{1\} \ \forall x,y \in \dom(\varphi) \quad d(\sigma h_1 \varphi(x),h_1 \varphi(y))=1$
\end{itemize}

Using the fact that $\Gamma_2$ is finite, acts strongly freely and $\U_S$ has the extension property, it is straightforward to build inductively a finite set $A$ and an isometry $\psi \colon \dom(\varphi) \to A$ such that
\begin{itemize}
\item $\forall h\in \Gamma_2\setminus \{1\} \ \forall a,a' \in A \ d(ha,a')=1$
\item $\forall a \in A \ \forall u\in \Gamma_2 \alpha(F) \quad d(u,a)=1$
\end{itemize}

Take $g_2\in \Gamma_2\setminus\Sigma$. Then the sets $\alpha(\Sigma F)$, $\Sigma A$, and $\Sigma g_2 A$ are pairwise disjoint. Define 
$$B=\Sigma F\sqcup \Sigma g_1 \dom(\varphi)\sqcup \Sigma h_1\rng(\varphi) \text{ and }C=\alpha(\Sigma F)\sqcup \Sigma A \sqcup \Sigma g_2A.$$ 
Note that $\Sigma B= B$, $\Sigma C=C$ and the bijection $\gamma_0\colon B\rightarrow C$ defined by $\gamma_0(u)=\alpha(u)$ for $u\in\Sigma F$, $\gamma_0(\sigma g_1 x)=\sigma \psi(x)$ and $\gamma_0(\sigma h_1 \varphi(x))=\sigma g_2 \psi(x)$, for all $\sigma\in\Sigma$ and for all $x \in \dom(\varphi)$ is, by construction, an isometry such that 
$\gamma_0\sigma=\sigma\gamma_0$ for all $\sigma\in\Sigma$. Since $\Sigma$ is finite, $\Sigma\curvearrowright\mathbb{U}_S$ is mixing so we may apply Proposition \ref{Extension} to get an extension $\gamma\in Z$ of $\gamma_0$. Note that $\gamma_{\restriction F}=\alpha_{\restriction F}$. Moreover, with $g=h_1^{-1}g_2g_1\in \Gamma$ and $x \in \dom(\varphi)$ we have 
$$\pi_\gamma(g)(x)=h_1^{-1}\gamma^{-1}g_2\gamma g_1x=h_1^{-1}\gamma^{-1}g_2\psi(x)=h_1^{-1}\gamma^{-1} \gamma h_1 \varphi(x)=\varphi(x) .$$ 
\end{proof}

\section{Actions of HNN extensions on bounded Urysohn spaces}

Let $\Sigma<H$ be a finite subgroup of a countable group $H$ and $\theta\colon \Sigma\rightarrow H$ be an injective group homomorphism. Define $\Gamma={\rm HNN}(H,\Sigma,\theta)$ the HNN-extension and let $t\in\Gamma$ be the ``stable letter'' i.e. $\Gamma$ is the universal group generated by $H$ and $t$ with the relations $t\sigma t^{-1}=\theta(\sigma)$ for all $\sigma\in\Sigma$. For $\epsilon\in\{-1,1\}$, we write
$$\Sigma_\epsilon:=\left\{\begin{array}{lcl}\Sigma&\text{if}&\epsilon=1,\\ \theta(\Sigma)&\text{if}&\epsilon=-1.\end{array}\right.$$

 Suppose that we have a faithful action $\Gamma\curvearrowright\UQone$ and view $\Gamma<\autUQone$. Define the closed (hence Polish space) subset $Z=\{\alpha\in\autUQone\colon \theta(\sigma)=\alpha\sigma\alpha^{-1}\text{ for all }\sigma\in\Sigma\}\subseteq\autUQone$ and note that it is non-empty (since $t\in Z$). By the universal property of $\Gamma$, for each $\alpha\in Z$ there exists a unique group homomorphism $\pi_\alpha\colon \Gamma\rightarrow\autUQone$ such that
$${\pi_\alpha}_{\restriction H}=\id_H\text{ and }\pi_\alpha(t)=\alpha.$$

In this section we prove the following result.

\begin{theorem}\label{ThmHNN}
Assume that $\Gamma\curvearrowright\UQone$ is free with the extension property,  $\Sigma_\epsilon\curvearrowright\UQone$ is strongly free, mixing and $\Sigma_\epsilon$ is highly core-free for $H\curvearrowright\UQone$ for all $\epsilon\in\{-1,1\}$. 
Then the set
$$O=\{\alpha\in\autUQone\colon \pi_\alpha\text{ is faithful and homogeneous}\}$$
is a dense $G_\delta$ in $Z$.
\end{theorem}

\begin{corollary}\label{cor: examples HNN}
For any subgroup $\Sigma$ of a countable group $H$ and any injective group homomorphism $\theta\,:\,\Sigma\rightarrow H$ such that $\Sigma_\epsilon<H$ is highly core-free for all $\epsilon\in\{-1,1\}$, we have $\text{HNN}(H,\Sigma,\theta)\in\mathcal{H}_{\UQone}$.
\end{corollary}

As is by now customary, we separate the proof in two lemmas.

\begin{lemma}\label{LemHNNHomogeneous}
Assume that for $\epsilon \in \{-1,1\}$ the action $\Sigma_\epsilon\curvearrowright\UQone$ is strongly free, mixing with the extension property and $\Sigma_\epsilon$ is highly core-free for $H\curvearrowright\UQone$. Then the set $U=\{\alpha\in Z \colon \pi_\alpha\text{ is homogeneous}\}$ is a dense $G_\delta$ in $Z$.
\end{lemma}

\begin{proof}
We apply the Baire category theorem as before: given $\varphi\in P_f(\UQone)$, $\alpha\in Z$ and $F\Subset \UQone$, we show that there exists $\gamma\in Z$ and $g\in\Gamma$ such that $\gamma_{\restriction F}=\alpha_{\restriction F}$ and $\pi_\gamma(g)_{\restriction \dom(\varphi)}=\varphi$. 

Since $\Sigma$ and $\theta(\Sigma)$ are highly core-free for $H\curvearrowright\UQone$, we produce $h_1,h_2\in H$ such that 
\begin{itemize}
\item $\forall u \in \Sigma F \ \forall x,x' \in \dom(\varphi) \  \forall \sigma \in \Sigma \setminus \{1\} \quad d(h_1x, u)=1=d(\sigma h_1x, h_1x') $.
\item $\forall u\in\alpha(\Sigma F\sqcup \Sigma h_1\dom(\varphi)) \ \forall y \in \rng(\varphi) \ d(h_2y, u)=1$.
\item $\forall y,y' \in \rng(\varphi) \ \forall \sigma \in \Sigma \setminus \{1\} \ d(\theta(\sigma) h_2y, h_2y')=1$.
\end{itemize}

Define $Y=\Sigma h_1 \dom(\varphi)\sqcup\Sigma\alpha^{-1}(h_2\rng(\varphi))\sqcup \Sigma F$ and observe that 
 $$\alpha(Y)=\theta(\Sigma) \alpha(h_1 \dom(\varphi))\sqcup\theta(\Sigma)(h_2\rng(\varphi))\sqcup(\theta(\Sigma)\alpha( F))$$
 Note that $\Sigma Y=Y$ and $\theta(\Sigma)\alpha(Y)=\alpha(Y)$. Define the bijection $\gamma_0\colon  Y\rightarrow\alpha(Y)$ by setting ${\gamma_0}_{\restriction \Sigma F}=\alpha_{\restriction \Sigma F}$ and 
$$\forall x \in \dom(\varphi) \ \forall \sigma \in \Sigma \quad \gamma_0(\sigma h_1x)=\theta(\sigma) h_2\varphi(x) \text{ and } \gamma_0(\sigma\alpha^{-1}(h_2\varphi(x)))=\theta(\sigma)\alpha(h_1x).$$
By the careful choices of elements $h_1,h_2$, the map $\gamma_0$ is an isometry such that $\gamma_0\sigma=\theta(\sigma)\gamma_0$ for all $\sigma\in\Sigma$. By Proposition \ref{Extension} there exists an extension $\gamma\in Z$ of $\gamma_0$. Note that $\gamma_{\restriction F}=\alpha_{\restriction F}$ moreover, with $g=h_2^{-1}th_1\in \Gamma$ we have, for all $x\in \dom(\varphi)$, 
$\pi_\gamma(g)x=h_2^{-1}\gamma h_1x=h_2^{-1}\gamma_0(h_1x)=h_2^{-1}h_2\varphi(x)=\varphi(x)$.
\end{proof}

\begin{lemma}\label{LemHNNfaithful}
Assume $\Gamma\curvearrowright\UQone$ is free, has the extension property and $\Sigma_\epsilon\curvearrowright\UQone$ is strongly free and mixing for all $\epsilon\in\{-1;1\}$. Then the set $V=\{\alpha \in Z \colon \pi_\alpha\text{ is faithful}\}$ is a dense $G_\delta$ in $Z$.
\end{lemma}

\begin{proof}
Once again we prove that $V_g=\{\alpha\in Z\colon \pi_\alpha(g)\neq\id\}$ is dense for all $g\in\Gamma\setminus \{1\}$, and we only need to prove it for $g\notin H$. 
Write $g=h_nt^{\epsilon_n}\dots t^{\epsilon_1}h_0$ a reduced expression for $g$, where $n\geq 1$, $h_k\in H$ and $\epsilon_k\in\{-1,1\}$. Fix $\alpha\in Z$ and $F\Subset \UQone$. 
For $l \in \{1,\ldots,n\}$ we defined subsets $H_l\subset\Gamma$:
$$H_1=\left\{\begin{array}{lcl}\Sigma h_0&\text{if}&\epsilon_1=1\\ \Sigma t^{-1} h_0&\text{if}&\epsilon_1=-1\end{array}\right.\text{ and, for }l\geq2,\,H_l=\left\{\begin{array}{lcl}\Sigma h_{l-1}t^{\epsilon_{l-1}}\dots t^{\epsilon_1}h_0&\text{if}&\epsilon_l=1\\ \Sigma t^{-1}h_{l-1}t^{\epsilon_{l-1}}\dots t^{\epsilon_1}h_0&\text{if}&\epsilon_l=-1\end{array}\right.$$
Observe that $\Sigma H_l=H_l$ for all $l$. Let $F':=F\cup\alpha(F)\Subset \UQone$ and, for $l \in \{1,\ldots,n\}$, $H'_l:=t H_l\subset\Gamma$. Using the extension property, we find $x\in \UQone$ such that $d(x, u)=1$ for all $u\in\Gamma F'$ and define
$$Y:=\Sigma F\sqcup(\sqcup_{l=1}^nH_lx)\quad\text{and}\quad Y':=\theta(\Sigma)\alpha(F)\sqcup(\sqcup_{l=1}^n H'_lx).$$
We may then define an isometry $\gamma_0\colon Y\rightarrow Y'$ by setting ${\gamma_0}_{\restriction \Sigma F}=\alpha_{\restriction \Sigma F}$ and, for all $l \in \{1,\ldots,n\}$, 
${\gamma_0}_{\restriction H_lx}=t_{\restriction H_lx}$. 

We have by construction $\Sigma Y=Y$, $\theta(\Sigma)Y'=Y'$ and $\gamma_0\sigma=\theta(\sigma)\gamma_0$ for all $\sigma\in\Sigma$. Thus there exists, by Proposition \ref{Extension}, an extension $\gamma\in Z$. Then $\gamma$ satisfies $\gamma_{\restriction F}=\alpha_{\restriction F}$ and $\pi_\gamma(g)x=h_n\gamma^{\epsilon_n}\dots\gamma^{\epsilon_1}h_0x=h_nt^{\epsilon_n}\dots t^{\epsilon_1}h_0x=gx\neq x$ since $g\neq 1$ and the $\Gamma$-action is free. It follows that $\gamma\in V_g$.
\end{proof}

\section{Actions of groups acting on trees on bounded Urysohn spaces}

We record here that, just as in \cite{zbMATH06503094}, the previous results apply to groups acting on trees. The reasoning is exactly the same as in \cite{zbMATH06503094}, but since the argument is short we kept it for the reader's convenience.

Let $\Gamma$ be a group acting without inversion on a non-trivial tree. By \cite{MR0476875}, the quotient graph $\mathcal{G}$ can be equipped with the structure of a graph of groups $(\mathcal{G},\{\Gamma_p\}_{p\in\VG},\{\Sigma_e\}_{e\in\EG})$ where each $\Sigma_e=\Sigma_{\overline{e}}$ is isomorphic to an edge stabilizer, each $\Gamma_p$ is isomorphic to a vertex stabilizer and such that $\Gamma$ is isomorphic to the fundamental group $\pi_1(\Gamma, \mathcal{G})$ of this graph of groups i.e., given a fixed maximal subtree $\mathcal{T}\subseteq\mathcal{G}$, the group $\Gamma$ is generated by the groups $\Gamma_p$ for $p\in\VG$ and the edges $e\in\EG$ with the relations
$$\overline{e}=e^{-1},\quad s_{e}(x)=er_{e}(x)e^{-1}\quad \forall x\in\Sigma_e\quad\text{and}\quad e=1 \quad\forall e\in {\rm E}(\mathcal{T}),$$
where $s_e\colon \Sigma_e\rightarrow \Gamma_{s(e)}$ and $r_e=s_{\overline{e}}\colon \Sigma_e\rightarrow\Gamma_{r(e)}$ are respectively the source and range group monomorphisms.

\begin{theorem}\label{Trees}
Assume $\Gamma_p$ is countably infinite for all $p\in\VG$, and $s_e(\Sigma_e)$ is highly core-free in $\Gamma_{s(e)}$ for all $e\in\EG$. Then $\Gamma\in\mathcal{H}_{\UQone}$.
\end{theorem}

\begin{proof}

Let $e_0$ be one edge of $\mathcal{G}$ and $\mathcal{G}'$ be the graph obtained from $\mathcal{G}$ by removing the edges $e_0$ and $\overline{e_0}$.

\textbf{Case 1: $\mathcal{G}'$ is connected.} It follows from Bass-Serre theory that $\Gamma={\rm HNN}(H,\Sigma,\theta)$ where $H$ is the fundamental group of our graph of groups restricted to $\mathcal{G}'$, $\Sigma=r_{e_0}(\Sigma_{e_0})<H$ is a subgroup and $\theta\colon \Sigma\rightarrow H$ is given by $\theta=s_{e_0}\circ r_{e_0}^{-1}$. By hypothesis $H$ is countably infinite and, since $\Sigma<\Gamma_{r(e_0)}$ (resp. $\theta(\Sigma)<\Gamma_{s(e_0)}$) is a highly core-free subgroup, $\Sigma<H$ (resp. $\theta(\Sigma)<H$) is also a highly core-free subgroup. Thus we may apply Theorem \ref{ThmHNN} to conclude that $\Gamma\in\mathcal{H}_{\UQone}$.

 \textbf{Case 2: $\mathcal{G}'$ is not connected.} Let $\mathcal{G}_1$ and $\mathcal{G}_2$ be the two connected components of $\mathcal{G}'$ such that $s(e_0)\in{\rm V}(\mathcal{G}_1)$ and $r(e_0)\in{\rm V}(\mathcal{G}_2)$. Bass-Serre theory implies that $\Gamma=\Gamma_1*_{\Sigma{e_0}}\Gamma_2$, where $\Gamma_i$ is the fundamental group of our graph of groups restricted to $\mathcal{G}_i$, $i=1,2$, and $\Sigma_{e_0}$ is viewed as a highly core-free subgroup of $\Gamma_1$ via the map $s_{e_0}$ and as a highly core-free subgroup of $\Gamma_2$ via the map $r_{e_0}$ since $s_{e_0}(\Sigma_{e_0})$ is highly core-free in $\Gamma_{s(e_0)}$ and $r_{e_0}(\Sigma_{e_0})$ is highly core-free in $\Gamma_{r(e_0)}$ by hypothesis. Since $\Gamma_1$ and $\Gamma_2$ are countably infinite and $\Sigma_{e_0}$ is highly core-free, we may apply Theorem \ref{ThmMain} to conclude that $\Gamma\in\mathcal{H}_{\UQone}$.
\end{proof}

\section{The unbounded case}\label{sec: unbounded}

In this section we explain how to extend some of the above results to the case of the Urysohn space $\mathbb U_\Q$. While the methods would apply to some other unbounded Urysohn spaces, (for instance they work also for $\U_\N$), we chose to focus on the case of the rational Urysohn space in order to improve the clarity of the exposition.

\subsection{Strong disconnection}

\begin{definition}
Let $D$ be an unbounded subset of $[0,+\infty[$.
An isometric action of $\Gamma$ on an unbounded metric space $(X,d)$ \textbf{strongly $D$-disconnects finite sets} if 
\begin{center}
$\forall F\Subset X$, $\exists N\in\N$, $\forall K\in D \cap [N,+\infty[$, $\exists \gamma\in\Gamma$ satisfying $d(x,\gamma y)=K$ for every $x,y\in F$.
\end{center}
\end{definition}

This definition might be a bit hard to grasp at first, as it involves many quantifiers. It is meant to capture the idea that one can find $\gamma \in \Gamma$ so that $\gamma F$ and $F$ are ``independent enough'', and in the unbounded case a good notion of independence is that all elements of $F$ and $\gamma F$ are at the same distance, and arbitrarily far away. We also need some freedom for the choice of distance, and this is what the set $D$ provides us (we can pick any distance in $D$, provided it is large enough).


As in the bounded case, we need to produce sufficiently free actions of a group $G$ on the Urysohn space; and these actions are built by starting from the left action of $G$ on itself and then applying a Kat\v{e}tov-type tower construction. A new feature here is that we first need to produce a suitable metric on $G$ before starting the tower construction.

\begin{lemma}\label{lem:unbounded left invariant metric}
Every infinite countable group admits an unbounded left-invariant metric which surjects onto $\N$.
\end{lemma}
\begin{proof}
If $\Gamma$ is finitely generated, let $S$ be a finite generating set, then the Cayley metric associated to $S$ is as wanted.

If $\Gamma$ is not finitely generated write $\Gamma=\bigcup_n\Gamma_n$ where $\Gamma_0=\{1\}$ and each $\Gamma_n$ is properly contained in $\Gamma_{n+1}$, then consider the left-invariant metric 
$d(g,h) = \min\{n\colon g^{-1}h \in \Gamma_n\}$. 

\end{proof}
\begin{definition}
Say that two metrics $d_1$ and $d_2$ on a set $X$ \textbf{coincide at scale} $K$ if we have 
$$\forall l \le K \quad \forall x,y\in X\quad d_1(x,y)=l \Leftrightarrow d_2(x,y)=l. $$
\end{definition}

\begin{lemma}
Suppose a countable group $\Gamma$ acts on a metric space $(X,d)$ with an unbounded orbit, and suppose that $d$ takes values in $\N$. Then for every $N\in\N$ there is a surjective function $f\colon\N\to\N$ so that if we let $\tilde d=f\circ d$, then $\tilde d$ is a metric  which coincides with $d$ at scale $N$ and the $\Gamma$-action on $(X,\tilde d)$ strongly $\N$-disconnects finite sets. 
\end{lemma}
\begin{proof}
Given a non-decreasing function $\varphi\colon\N\setminus\{0\}\to\N \setminus\{0\}$, we can associate to it a function $f_\varphi\colon\N\to\N$ by defining 
$$f_\varphi(m)= \min \{k \colon m\leq\sum_{i=1}^{k}\varphi(i) \} . $$
It can be checked that  $f_\varphi$ is the unique non-decreasing function $\N\to\N$ such that $f(0)=0$ and for every $n\in\N \setminus\{0\}$ we have $\abs{f_\varphi^{-1}(\{n\})}=\varphi(n)$. 
Using the fact that $\varphi$ is non-decreasing, it is straightforward to check that $f_\varphi$ is subadditive. We then see that $f_\varphi\circ d$ is still a metric as soon as $d$ is a metric with values in $\N$. Also observe that if $\varphi(n)=1$ for all $n\leq N$, then $f_\varphi(n)=n$ for all $n\leq N$. We then set $d_\varphi:=f_\varphi\circ d$.

We use this construction twice; begin by picking $x_0$ with an unbounded orbit. We first set $\varphi(i)=1$ for all $i \le N$. Then, we inductively define $\varphi(N+j)$ and find elements $\gamma_k \in \Gamma$ in such a way that (defining the empty sum as being equal to $0$)
$$\forall k\ge 0 \quad N+ \sum_{j=1}^k \varphi(N+j) < d(x_0,\gamma_k x_0) < N+ \sum_{j=1}^{k+1} \varphi(N+j).  $$ 
Then $d$ and $d_\varphi$ coincide at scale $N$, and $f_\varphi(d(x_0,\gamma_k x_0))=N+k+1$ for all $k \ge 0$. In particular, for every $n> N$ there exists some $\gamma\in\Gamma$ such that $d_\varphi(x_0,\gamma x_0)=n$. So we may as well assume that $d$ already satisfies this surjectivity condition. 

We then let $\varphi(n)=1$ for $n\leq N$ and $\varphi(n)=2n$ for $n \ge N+1$. Let us show why $d_\varphi$ is now as wanted.

Pick $F \Subset \Gamma$, without loss of generality we assume $x_0\in F$. Let $N' > N$ be such that $N'>2\mathrm{diam}_d(F)$.  By our surjectivity assumption on $d$ there is $\gamma\in\Gamma$ such that $$d(x_0,\gamma x_0)=N'+\sum_{i=1}^{N'-1} \varphi(i).$$
Then for every $x,y\in F$ we have 
\begin{align*}
\abs{d(x,\gamma y)-d(x_0,\gamma x_0)}&\leq \abs{d(x,\gamma y)-d(x,\gamma x_0)}+\abs{d(x,\gamma x_0)-d(x_0,\gamma x_0)}\\
&\leq d(\gamma y,\gamma x_0)+d(x,x_0)\\
&\leq 2\mathrm{diam}(F)\\ &<N'.
\end{align*}
So since $d(x_0,\gamma x_0)=N'+\sum_{i=1}^{N'-1} \varphi(i)$ we conclude that 
$$\sum_{i=1}^{N'-1}\varphi(i)< d(x,\gamma y)<\sum_{i=1}^{N'-1}\varphi(i)+2N'=\sum_{i=1}^{N'}\varphi(i) ,$$ 
hence $d_\varphi(x,\gamma y)=N'$ as wanted.
\end{proof}

Note that by rescaling, the above lemma applies for any metric taking values in $\alpha\N$ for some $\alpha>0$.

\begin{proposition}\label{p:unboundedfreeaction}
Every countable group admits a free isometric action on $\U_\Q$ which strongly $\N$-disconnects finite sets.
\end{proposition}

In the proof we make the natural unbounded modification of extensions with parameters (Section \ref{sec: ext with params}): if we are given a metric space $(Y,d)$, a \emph{nonempty} subset $X\subseteq Y$ and a set $A$, we let $(Y\times_XA,d)$ be the quotient metric space obtained from the pseudo-metric on $Y\times A$ given by
$$d((y,a),(y',a'))=\left\{
\begin{array}{cl}
\inf_{x\in X} d_Y(y,x)+d_Y(x,y')&\text{ if }a'\neq a, \\
d_Y(y,y') &\text{ if }a'=a.
\end{array}
\right.
$$
\begin{proof}[Proof of Proposition \ref{p:unboundedfreeaction}]
Start with $\Gamma$ acting on itself by isometries for an unbounded left-invariant metric provided by lemma \ref{lem:unbounded left invariant metric}, denote it $X_0$ and observe that this action is free. Then consider the set $E_{1/2}(X_0)$ of finitely supported Kat\v{e}tov extensions of $X_0$ taking values in $\frac12\N$ and equip it with the amalgam distance $d(f,g)=\inf_x f(x)+g(x)$. Then let $X_1=E_{1/2}(X_0)\times_{X_0}\Gamma$. The diagonal action of $\Gamma$ on $(X_1,d)$ is still free and has an unbounded orbit since $X_0$ embeds isometrically, but we may lose strong disconnection. 

So we let $d_1$ be a metric provided by the previous lemma with $N=2$. At stage $n$, $(X_n,d_n)$ being constructed, we consider $X_{n+1}=E_{1/n!}(X_n)\times_{X_n}\Gamma$ with a metric $d_{n+1}$ which is the same as $d_n$ at scale $2^{n+1}$ but for which the diagonal $\Gamma$-action strongly $\N$-disconnects finite sets, which is possible via the previous lemma. 

For $x,y\in\bigcup_n X_n=:X_\infty$, since $2^{n+1}\to\infty$ and $d_{n+1}$ coincides with $d_n$ at scale $2^{n+1}$,   the sequence $d_n(x,y)$ is stationary and we let $d(x,y)$ be the limit. It is easy to check that $d$ is still a metric.

Furthermore, $X_\infty$ has the extension property for $\Q$-valued metrics. Indeed if we have a $\Q$-valued Kat\v{e}tov function on some finite $F\subseteq X_\infty$, we take $n$ large enough so that $2^n\geq\max(\diam(F),\norm f_\infty)$, $f$ takes values in $\frac 1{n!}\N$ and $F$ is contained in $X_n$. We then have that $f$ is realised in the metric space $E_{1/n!}(X_n)$. Since the latter embeds in $X_{n+1}$ in a way which preserves the metric at scale $2^{n+1}$, $f$ is actually realised in $X_{n+1}$ and thus in $X_\infty$ since the inclusion $(X_{n+1},d_{n+1})\subseteq (X_\infty,d)$ preserves the metric at scale $2^{n+1}$. 

Let us finally show that the $\Gamma$-action on $X_\infty$ strongly $\N$-disconnects finite sets. Take $F\subseteq X_\infty$ finite, let $n\in\N$ such that $F\subseteq X_n$. Since the $\Gamma$-action on $(X_n,d_n)$ strongly $\N$-disconnects finite sets, we find $N\in\N$ such that for any integer $k\geq N$, there is $\gamma\in\Gamma$ such that $d_n(x,\gamma y)=k$ for all $x,y\in F$. We may as well assume $N\geq 2^{n+1}$. 

Now let $k\geq N$. Observe that the definition of the metrics $d_k$ implies that for every $l\geq N$, there is $\gamma\in\Gamma$ such that $d_l(x,\gamma y)=k$ for all $x,y\in F$ (indeed their restrictions to $X_n$ are obtained by composing $d_n$ with finitely many surjective maps $\N\to\N$).

Let $l$ such that $2^{l+1}>k$, there is $\gamma\in\Gamma$  such that $d_l(x,\gamma y)=k$ for all $x,y\in F$. Since $d_l$ and the restriction of $d$ to $X_l$ coincide at scale $2^{l+1}$, this implies that 
$d(x,\gamma y)=k$ for all $x,y\in F$ as wanted.
\end{proof}

\subsection{Dense free products}
Let $\Gamma$ and $\Lambda$ be two countable infinite groups acting faithfully on $\mathbb U_\Q$ and view $\Gamma,\Lambda<{\rm Iso}(U_\Q)$. Then, for all $\alpha\in\Iso(\mathbb U_\Q)$, there exists a unique group homomorphism $\pi_\alpha\,:\,\Gamma*\Lambda\rightarrow \Iso(\mathbb U_\Q)$ such that
$$\pi_\alpha(g)=\left\{\begin{array}{lcl} g&\text{if}& g\in\Gamma,\\\alpha^{-1}g\alpha&\text{if}& g\in\Lambda.\end{array}\right.$$
\begin{theorem}
Let $\Gamma$ and $\Lambda$ be two countable infinite groups acting freely on $\mathbb U_\Q$. Suppose both actions strongly $\N$-disconnect finite sets. Then the set of $\alpha\in\Iso(\mathbb U_\Q)$ so that $\pi_\alpha$ is faithful and homogeneous is dense $G_\delta$.
\end{theorem}

As usual we decompose the proof in two parts. 
\begin{lemma}\label{lem: free prod homogeneous for unbounded Urysohn}
Let $\Gamma$ and $\Lambda$ be two countable infinite groups acting on $\mathbb U_\Q$, suppose their actions strongly $\N$-disconnect finite sets. Then the set of $\alpha\in\Iso(\mathbb U_\Q)$ so that $\pi_\alpha$ is homogeneous is dense $G_\delta$.
\end{lemma}

\begin{proof}
Let $\varphi$ be a finite partial isometry, $\beta\in\Iso(\mathbb U_\Q)$, $F$ a finite subset of $\mathbb U_\Q$. We must find $\alpha$ coinciding with $\beta$ on $F$ such that for some $g\in\Gamma*\Lambda$ we have $\pi_\alpha(g)_{\restriction \dom \varphi}=\varphi$. 

We first find $\gamma_1\in\Gamma$ such that every element of $F\cup\beta F$ is at fixed distance $K$ from every element of $\gamma_1\dom\varphi$. Let $F'=F\cup\beta F\cup\gamma_1\dom\varphi$.

We then find $\gamma_2\in\Gamma$ and $\lambda\in\Lambda$ such that every element of $\gamma_2^{-1}\rng\varphi \cup \lambda\gamma_1\dom\varphi$ is at distance $K'$ from every element of $F'$ (this is where we use that strong disconnection works for \emph{any} large enough $K' \in \N$: this enables us to find $\gamma_2$ and $\lambda$ simultaneously)

Finally, define the finite partial isometry $\alpha$ with domain $F\sqcup\gamma_1\dom\varphi\sqcup\gamma_2^{1}\rng\varphi$ by $\alpha(x)=\beta(x)$ $\forall x\in F$, $\alpha(x)=x$ $\forall x\in\gamma_1\dom\varphi$, and $\alpha(\gamma_2^{-1}\varphi(x))=\lambda\gamma_1x$ $\forall x\in\dom\varphi$ (see Fig. \ref{fig: alpha unbounded case}).

\begin{figure}[!h]\
\includegraphics[scale=1]{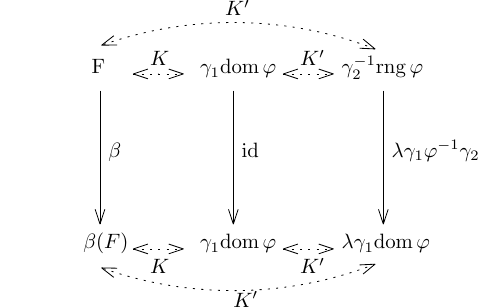}

\caption{$\alpha$ is a partial isometry.}
\label{fig: alpha unbounded case}
\end{figure}
We now extend $\alpha$ to an isometry of $\mathbb U_\Q$ which we still denote by $\alpha$. By construction $\alpha$ extends the restriction of $\beta$ to $F$. Moreover if we let $g=\gamma_2\lambda\gamma_1$, we have for all $x\in \dom\varphi$
$$ \pi_\alpha(g) x = \gamma_2\alpha^{-1}\lambda\alpha\gamma_1x =\gamma_2\alpha^{-1}\lambda\gamma_1x =\gamma_2\gamma_2^{-1}\varphi(x)=\varphi(x),$$
as wanted.
\end{proof}

\begin{lemma}
Let $\Gamma, \Lambda$ be two groups acting freely on $\U_\Q$, and $w$ be an element of $\Gamma *  \Lambda$ which does not belong to $\Gamma$. Then, for any $\beta \in  \Iso(\mathbb U_\Q)$ the set
$W_f=\{\alpha \colon  \pi_\alpha(w)=\beta \} $ does not contain $1$ in its interior.
\end{lemma}

\begin{proof}
We use the fact that for any finite subset $A \Subset \U_\Q$, any $\alpha \in \Iso(A)$ such that $\alpha^2=1$, and any $x \not \in A$, there exist infinitely many $y \in \U_\Q$ such that one can extend $\alpha$ to an isometric involution of $A \cup\{x,y\}$ by setting $\alpha(x)=y,  \alpha(y)=x$.

Fix an open neighborhood $U$ of $1$, which we may assume is of the form 
$$U=\{\alpha \colon  \forall x \in F \ \alpha(x)=x\}$$ for some finite subset $F$ of $\U_\Q$. Let $w=\lambda_1\gamma_2\ldots \gamma_n \lambda_n$ be the reduced form of $w$ (one can reduce to the case where $w$ has such a reduced form by multiplying $\beta$ on the left and/or the right by an element of $\Gamma$). Below it will be useful to write $\gamma_1=1$.

For $\alpha \in \Iso(\U_\Q)$, we write 
$$w(\alpha)= (\alpha \lambda_1 \alpha) \gamma_2 (\alpha \lambda_2 \alpha) \gamma_3 \ldots \gamma_n (\alpha \lambda_n \alpha)  $$
Our aim is to find an involution $\alpha \in U$ such that $w(\alpha) \ne \beta$, since this will imply that $U$ is not contained in the interior of $W_f$. We begin by picking $x \in \U_Q \setminus F$, and find $y \in \U_\Q \setminus (F\cup \lambda_n^{-1}(F) \cup \{x,\lambda_n^{-1}x, \beta(x), \lambda_n^{-1}\beta(x)\})$ such that setting 
$$\forall z \in F \ \alpha(z)=z \ , \quad \alpha(x)=y \ , \quad \alpha(y)=x $$
defines an isometric involution $\alpha$ of $F \cup \{x,y\}$. We now proceed inductively 
to extend $\alpha$ to a partial isometric involution such that the elements
$\alpha x, \lambda_n \alpha x, \alpha \lambda_n \alpha x,\gamma_n \alpha \lambda_n \alpha x, \ldots, \alpha \lambda_1 \alpha \gamma_2 \ldots \gamma_n \alpha \lambda_n \alpha x$ are all defined, pairwise distinct, and do not belong to $F \cup \{x,\beta(x)\}$. Then any extension of $\alpha$ to an isometric involution of $\U_\Q$ is an element of $U$ such that $w(\alpha) \ne\beta $. During the inductive process, we have to deal with two different cases:
\begin{enumerate}
\item We want to define $\alpha (\lambda_k \ldots \lambda_n \alpha x)$ for some $k \ge 1$ (note that our inductive conditions ensure that $\lambda_k \ldots \lambda_n \alpha x$ is not in the domain of $\alpha)$. Set 
$$B= \{x,\alpha x, \lambda_n \alpha x, \alpha \lambda_n \alpha x, \gamma_n \alpha \lambda_n \alpha x,\ldots, \lambda_k \ldots \lambda_n \alpha x\} \cup \{\beta(x)\}$$
Then, we can find an element $z$ which is not in $F\cup B \cup \gamma_k^{-1}(B) \cup \gamma_k^{-1}(F)$ such that setting 
$$\alpha (\lambda_k \ldots \lambda_n \alpha x)=z \ , \quad \alpha(z)= \lambda_k \ldots \lambda_n \alpha x $$ 
is our desired extension of $\alpha$ (note that in our inductive conditions we require also $\gamma_k z \ne z$ for $k \ge 2$, but this is automatic since nontrivial elements of $\Gamma$ do not have any fixed points).
\item We want to define $\alpha (\gamma_k \alpha \lambda_k \ldots \lambda_n \alpha x)$ for some $k \ge 2$: apply the same reasoning as above, replacing $\gamma_k$ with $\lambda_{k-1}$ and $B$ with 
$$ \{x,\alpha x, \lambda_n \alpha x, \alpha \lambda_n \alpha x, \gamma_n \alpha \lambda_n \alpha x,\ldots, \gamma_k \alpha \lambda_k \ldots \lambda_n \alpha x\} \cup \{\beta(x)\} \ .$$
\end{enumerate} 
\end{proof}

\begin{lemma}\label{lem: free prod faithful for unbounded Urysohn}
Let $\Gamma$ and $\Lambda$ be two countable  groups acting freely on $\mathbb U_\Q$. Then the set of $\alpha\in\Iso(\mathbb U_\Q)$ so that $\pi_\alpha$ is faithful is dense $G_\delta$.
\end{lemma}
\begin{proof}
Using the Baire category theorem, it is enough to show that for any $w \in \Gamma * \Lambda \setminus \{1\}$, the set 
$\{\alpha \colon  \pi_\alpha(w)=1\}$ has empty interior. This is true by assumption if $w$ belongs to either $\Gamma$ or $\Lambda$; if this result is false for some $w$, then we have some $\alpha \in \Iso(\U_\Q)$ and a neighborhood $V$ of $1$ such that $\pi_{\alpha \beta}(w)=1$ for any $\beta \in V$.

We can write $w= \gamma_1\lambda_1 \ldots \gamma_n \lambda_n$, with at least one $\lambda_i$ different from $1$. Since
\begin{align*}
\pi_{\alpha \beta}(w) &=  \gamma_1 \lambda_1^{\alpha \beta} \ldots \gamma_n \lambda_n^{\alpha \beta} \\
                      &=  (\gamma_1 \alpha) \lambda_1^{\beta} \gamma_2^{\alpha^{-1}} \lambda_2^\beta \ldots \gamma_n^{\alpha^{-1}}\lambda_n^{\beta}\alpha^{-1}
\end{align*}
we see that $\lambda_1^\beta \gamma_2^{\alpha^{-1}} \lambda_2^\beta \ldots \gamma_n^{\alpha^{-1}}\lambda_n^{\beta}= \alpha^{-1} \gamma_1^{-1} \alpha$ for all $\beta \in V$ and this contradicts our previous lemma (since the conjugate of the action of $\Gamma$ by $\alpha^{-1}$ is a free action of $\Gamma$).
\end{proof}

\begin{corollary}\label{CorUnbounded}
Every free product of infinite groups admits a faithful homogeneous action on the rational Urysohn space.
\end{corollary}

\section{ An example: the group of finitely supported permutations}
Having spent quite some time building dense subgroups of isometry groups of countable Urysohn spaces, we now are naturally led to the question: how does the class of dense subgroups of $\Iso(\U_S)$ depend on the distance set $S$? We know embarrassingly little about this problem, but we do know that the case where $\Iso(\U_S)=\mathfrak S_\infty$ is different from the others.


Recall that $\mathfrak S_{(\infty)}$ is the (countable) group of all permutations of $\N$ with finite support. By studying its primitive actions, we will prove the 
following result.

\begin{theorem}\label{t:anexample}Let $S$ be a bounded or unbounded distance set.
The group $\mathfrak S_{(\infty)}$ embeds densely in $\Iso(\U_S)$ if and only if $|S|=2$.
\end{theorem}
Of course when $|S|=2$, $\Iso(\U_S)$ is isomorphic to $\mathfrak S_\infty$ so only one implication above is interesting.

The above theorem shows that the classes of countable groups which are isomorphic to a dense subgroup of $\mathfrak S_\infty$ and, say, the  automorphism group $\text{Aut}(R)$ of the random graph, are not the same; much remains to be investigated. For instance, we do not know of a countable dense subgroup of $\text{Aut}(R)$ which is not isomorphic to a dense subgroup of $\mathfrak S_\infty$, though it seems likely that such groups exist.

Our approach is fairly elementary: assume that $\mathfrak S_{(\infty)}$ acts faithfully and homogeneously on $\U_S$. Pick $x \in \U_S$, and consider the associated stabilizer subgroup
$$\Delta =\{\gamma \in \mathfrak S_{(\infty)} \colon  \gamma x = x \}. $$
Since $\Iso(\U_S)$ acts transitively on $\U_S$, we can identify $\mathfrak S_{(\infty)} /\Delta$ with $\U_S$, and the action of $\mathfrak S_{(\infty)}$ on $\U_S$ is simply the left-translation action of $\mathfrak S_{(\infty)}$ on $\mathfrak S_{(\infty)} /\Delta$. Moreover, since the action is faithful the subgroup $\Delta$ is core-free and, since the action is homogeneous, the closure of $\mathfrak S_{(\infty)}$ in the symmetric group of $\mathfrak S_{(\infty)} /\Delta$ is isomorphic to $\Iso(\U_S)$.  

Conversely, if there is a core-free subgroup $\Delta$ such that the closure of $\mathfrak S_{(\infty)}$ in the symmetric group of $\mathfrak S_{(\infty)} /\Delta$ is isomorphic to $\Iso(\U_S)$, then obviously $\mathfrak S_{(\infty)}$ is isomorphic to a dense subgroup of $\U_S$. Thus, we focus our attention towards understanding what kind of subgroups $\Delta$ can arise as point stabilizers.

\subsection{Classification of point stabilizers for homogeneous actions}

We recall some facts and definitions from permutation group theory.

\begin{definition}
Let $X$ be a countable set, and $G$ be a group acting transitively on $X$. A \emph{block} for this action is a nonempty set $A \subseteq X$ such that for all $g \in G$ one has either $gA=A$ or $gA \cap A = \emptyset$. A block is \emph{trivial} if it is either a singleton or the whole set $X$.

The action of $G$ on $X$ is \emph{primitive} if all blocks for this action are trivial (equivalently, there is no nontrivial $G$-invariant equivalence relation on $X$).
\end{definition}

This property is particularly relevant to us because of the following two facts (the first of which is a standard fact in permutation group theory, while the second is probably well-known. We include the short proofs for the reader's convenience).

\begin{lemma}\label{prop: trans prim implies max}
Assume that $G$ is a group that acts transitively on $X$. Then the action is primitive if and only if the stabilizer of some (any) $x \in X$ is a maximal proper subgroup.
\end{lemma}

\begin{proof}
Given $x \in X$, denote by $H$ its stabilizer. We can identify the action of $G$ on $X$ with the regular action of $G$ on $G/H$. If there exists a subgroup $K$ of $G$ such that $H \subsetneq K \subsetneq G$, then the $K$-cosets form a family of nontrivial blocks for this action. Conversely, assume that $A \subseteq G/H$ is a block for this action, containing $H$ (seen as an element of $G/H$). By definition of a block, we see that 
$$K:=\{g \in G \colon  gH \in A\} =\{g \in G \colon  gA=A\}$$ 
so $K$ is a subgroup of $G$, and if $A$ is nontrivial then $H \subsetneq K \subsetneq G$.
\end{proof}

\begin{lemma}\label{prop: Iso acts transitive primitive}
  The action of $\Iso(\U_S)$ on $\U_S$ is transitive and primitive.
\end{lemma}

\begin{proof}
Transitivity of the action is obvious. As for primitivity, assume that $A$ is a block for this action, with two elements $x,y$ such that $d(x,y)=r >0$. For any $z$ such that $d(y,z)=r$, there exists an isometry $\alpha$ of $\U_S$ such that $\alpha(y)=y$ and $\alpha(x)=z$. Since $A$ is a block, $\alpha(y)=y$ implies $\alpha(A)=A$ so $z$ belongs to $A$. Now, note that for any $s \in S \cap [0,2r]$ there exists $z \in \U_S$ with $d(x,z)=s$ and $d(y,z)=r$, so $A$ has an element $z$ such that $d(x,z)=s$. Switching the roles of $x$ and $y$, we conclude that $A$ contains all $z$ such that $d(x,z) \in [0,2r]$. From this we deduce that $A=\U_S$.
\end{proof}

It is straightforward to check that transitivity and primitivity are inherited by dense subgroups, so we have the following result.

\begin{proposition}\label{prop: dense in iso implies maximal proper}
If $\Gamma$ is a dense subgroup of $\Iso(\U_S)$, then for any $x \in \U_S$ its stabilizer $\Gamma_x$ is a maximal proper subgroup of $\Gamma$.
\end{proposition}

In particular, if $\mathfrak S_{(\infty)}$ acts faithfully homogeneously and $\Delta$ is the stabilizer of a point, then $\Delta$ is a maximal proper subgroup of $\mathfrak S_{(\infty)}$. We now turn our attention towards understanding these maximal subgroups; probably the description we obtain is well-known, but we do not know of a reference. 

\begin{definition}
Assume that $H$ is a group that acts transitively and not primitively on a set $X$.
We say that $H$ is \emph{almost primitive} if there exists a maximal nontrivial block, and \emph{totally imprimitive} otherwise.
\end{definition}

In the classification of maximal subgroups of $\mathfrak S_{(\infty)}$, it is natural to distinguish groups with finite/infinite biindex, a notion that we introduce now.

\begin{definition}
The \textbf{biindex} of a subgroup $H$ of a group $G$ is the minimal cardinality of a set $F$ such that $HFH=G$, that is, the number of double cosets. 
\end{definition}

It is straightforward to check that the biindex of $H$ corresponds to the number of orbits for the $H$-action on $G/H$. 

Maximal subgroups of $\mathfrak S_{(\infty)}$ which do not act transitively on $\N$ are easily described.

\begin{lemma}\label{lemma: some maximal subgroups}
Let $X$ be a subset of $\N$ different from $\emptyset$ and $\N$. Then the setwise stabilizer of $X$ is a maximal proper subgroup of $\mathfrak S_{(\infty)}$.
\end{lemma}
\begin{proof}
Since the setwise stabilizer of $X$ is equal to the setwise stabilizer of its complement, we may as well assume that $X$ is infinite.

Given $x_0, x_1\in \N$ the transposition $(x_0\enskip x_1)$ is the permutation $\sigma$ defined by $\sigma(x_0)=x_1$, $\sigma(x_1)=x_0$ and $\sigma(x)=x$ for all $x\not\in\{x_0,x_1\}$. 
We use the fact that if $G\subseteq\N\times\N$ is a graph on $\N$, then the set of transpositions of the form $(x\enskip y)$ with $(x,y)\in G$ generates $\mathfrak S_{(\infty)}$ if and only if the graph $G$ is connected.

Let $\Delta$ be the setwise stabilizer of $X$, and let us show $\Delta$ is a maximal subgroup of $\mathfrak S_{(\infty)}$. 
Observe that the graph of $(x,y)$ such that $(x\enskip y)\in \Delta$ contains both the complete graph on $X$ and the complete graph on $\N\setminus X$.

Now let $g\in \mathfrak S_{(\infty)}\setminus\Delta$. Then we find $x\in X$ such that $g(x)\not\in X$. 
Since $g$ has finite support and $X$ is infinite, we find $y\in X$ with $x\neq y$ and $g(y)=y$. Now $g(x\enskip y)g\inv=(g(x)\enskip y)$. We conclude that the graph of $(x,y)$ such that $(x\enskip y)\in \la \Delta, g\ra$ is connected, which shows that $\la \Delta, g\ra=\mathfrak S_{(\infty)}$ as wanted. 
\end{proof}

Using a famous theorem of Wielandt, it is also not difficult to identify maximal proper subgroups of $\mathfrak S_{(\infty)}$ which act transitively on $\N$.

\begin{lemma}\label{l:classification}
Let $\Delta$ be a maximal proper subgroup of $\mathfrak S_{(\infty)}$, and assume that $\Delta$ has infinite index and acts transitively on $\N$. Then there is an equivalence relation $E$ on $\N$, all of whose classes have the same cardinality $k \ge 2$ (possibly infinite), and such that $\Delta$ consists of all elements of $\mathfrak S_{(\infty)}$ which preserve $E$. These groups are indeed maximal, and have infinite biindex.
\end{lemma}
\begin{proof}
Suppose otherwise.
If $\Delta$ is primitive, $\Delta$ must be either $\mathfrak S_{(\infty)}$ or the group of even permutations by Wielandt's theorem \cite[Thm. 3.3D]{zbMATH00894528}, contradicting the assumption on the index of $\Delta$.

If $\Delta$ is almost primitive, let $A$ be a maximal nontrivial block for the action $\Delta$ on $\N$. Note that all the maximal blocks are of the form $\delta A$, $\delta \in \Delta$, and denote by $E$ the equivalence relation  defined by 
$$\forall i,j \in \N \quad (i E j )\Leftrightarrow (\exists \delta \in \Delta \  \delta i \in A \text{ and } \delta j \in A ). $$
Then all the $E$-equivalence classes have the same cardinality $|A| \ge 2$. By maximality, $\Delta$ must coincide with the group of all elements of $\mathfrak S_{(\infty)}$ which are automorphisms of $E$. 

Say that a block $\delta A$ is perturbed by $\gamma \in \mathfrak S_{(\infty)}$ if $\gamma \delta A$ is not equal to some $\delta'A, \delta' \in \Delta$. Let $N(\gamma)$ be the (finite) number of blocks $\delta A$ which are perturbed by $\gamma$. Since $\Delta$ respects the equivalence relation $E$, any element of $\Delta \gamma \Delta$ must perturb as many blocks as $\gamma$; there are elements of $\mathfrak S_{(\infty)}$ which perturb an arbitrary number of blocks, so the biindex of $\Delta$ in $\mathfrak S_{(\infty)}$ is infinite.

If $\Delta$ is totally imprimitive, we get an increasing sequence of blocks $A_1 \subseteq A_2 \subseteq \ldots$ whose union is equal to $\N$ . By maximality, each element with support  in some $A_i$ must belong to $\Delta$, so $\Delta=\mathfrak S_{(\infty)}$, a contradiction (i.e.~there is no proper, maximal,transitive, totally imprimitive subgroup of $\mathfrak S_{(\infty)}$).

It remains to prove that groups of finitary permutations which preserve equivalence relations whose classes have fixed cardinality $k\geq 2$ ($k$ can be equal to $\aleph_0$ here) are indeed maximal. Let $E$ be an equivalence relation all of whose classes have fixed cardinality $k\geq 2$, let $\Delta$ be its automorphism group. Pick $\sigma \in \mathfrak S_{(\infty)}\setminus \Delta$ and
consider the graph $G$ of $(x,y)$ such that $(x\enskip y)\in \langle \Delta , \sigma \rangle$. As in the proof of the previous lemma, it suffices to show that this graph is connected. First, note that each $E$-class is contained in a connected component. 
Since $\sigma \not \in \Delta$, there exist $x,y\in\N$ with $(x,y)\in E$ such that $(\sigma (x),\sigma (y))\not\in E$. Now observe that for every $z\in\N\setminus [\sigma(x)]_E$ there is $\tau\in \Delta$ such that $\tau(\sigma (y))=z$ and $\tau(\sigma(x))=\sigma(x)$. Conjugating the transposition
$(x\enskip y)$ by $\tau\sigma$, we conclude that $(\sigma(x),z)\in G$, which proves that $G$ is connected as wanted. 
\end{proof}

We can now list the maximal proper subgroups of $\mathfrak S_{(\infty)}$; from this description we will then be able to understand the corresponding Schlichting completions.

\begin{theorem}\label{thm: classification maximal proper infinite index}
Assume that $\Delta$ is a maximal proper subgroup of $\mathfrak S_{(\infty)}$, of infinite index. Then, either:
\begin{enumerate}
\item \label{eq: finite biindex transitive}
$\Delta$ has finite biindex $k+1$, in which case $\Delta$ is the setwise stabilizer of a finite set of integers with cardinality $k$. 

\item \label{eq:infinite biindex not transitive}
$\Delta$ has infinite biindex and does not act transitively; then it is the setwise stabilizer of an infinite, coinfinite set of integers.

\item \label{eq:infinite biindex transitive}
$\Delta$ has infinite biindex, and acts transitively; then there exists an equivalence relation $E$, all of whose classes have the same (possibly infinite) cardinality $k \ge 2$, such that $\Delta$ consists of the elements of $\mathfrak S_{(\infty)}$ which are automorphisms of $E$.  
\end{enumerate}
All the groups described above are indeed maximal.
\end{theorem}

\begin{proof}
Assume first that $\Delta$ does not act transitively on $\N$. By maximality, we must have exactly two disjoint $\Delta$-orbits, say $X$ and $Y$; at least one of them (say $Y$) is infinite. By maximality, $\Delta$ is the setwise stabilizer of $X$, which is indeed maximal by Lemma \ref{lemma: some maximal subgroups}. 

If $X$ is infinite then $\Delta$ is the stabilizer of a point for the action of $\mathfrak S_{(\infty)}$ on the set $\Omega$ of all infinite and co-infinite subsets of $\N$. It follows that $\Delta$ has infinite biindex.

If $\Delta$ has finite biindex $k+1$, then $X$ is finite of cardinality $K\in\N$ and $\Delta$ is the stabilizer of a point for the (transitive) action of $\mathfrak S_{(\infty)}$ on the set $[\N]_K$ of all subsets of $\N$ whose cardinality is equal to $K$. Then the biindex of $\Delta$ is equal to $K+1$ (indeed the $\Delta$-orbit of $Y\in [\N]_K$ is determined by the cardinality of its intersection with $X$) and we conclude $k=K$ as wanted. 

The case where $\Delta$ acts transitively has already been described in the previous lemma.
\end{proof}

\subsection{Associated Schlichting completions}

As explained at the beginning of this section, we want to understand the closure of $\mathfrak S_{(\infty)}$ in  the symmetric group of $\mathfrak S_{(\infty)})/\Delta$ when $\Delta$ is one of the subgroups from the above proposition, in other words we want to understand possible \emph{Schlichting completions} of $\mathfrak S_{(\infty)}$ with respect to maximal subgroups (see the definition below). It is not necessary to go into that much detail if one simply wants to prove that $\mathfrak S_{(\infty)}$ is not isomorphic to a dense subgroup of a nontrivial Urysohn space, but it seems interesting to work out this example in complete detail.

\begin{definition}
Let $\Delta$ be a subgroup of  a countable group $\Gamma$. The \textbf{Schlichting completion} of $\Gamma$ with respect to $\Delta$ is the closure of $\Gamma$ in the symmetric group of $\Gamma /\Delta$.
\end{definition}

This notion was introduced by Schlichting when $\Delta$ was \emph{commensurated} in $\Gamma$, and he proved that the completion is then a locally compact group \cite{MR583752}.
 
Note that the Schlichting completion of a group $\Gamma$ comes with a canonical map from $\Gamma$  with dense image. This map will always be clear from the context, so we will not explicitely write it down.

We will now define concrete representations of the Schlichting completions of $\mathfrak S_{(\infty)}$ with respect to the subgroups that we found in the above section. We will often make use of the following fact without further mention: if $G$ and $H$ are  Polish groups and $\pi\colon G\to H$ is an embedding of topological groups, then $\pi$ has closed image (see \cite[Prop. 2.2.1]{MR2455198}).

Let us start with item \eqref{eq: finite biindex transitive} from Theorem \ref{thm: classification maximal proper infinite index}.

\begin{proposition}\label{prop: completion maximal finite biindex}
Let $\Delta$ be a maximal proper subgroup of $\mathfrak S_{(\infty)}$, and assume that $\Delta$ has infinite index and biindex $k+1$. Then the Schlichting completion of $\mathfrak S_{(\infty)}$ with respect ot $\Delta$ is $\mathfrak S_\infty$. In particular, $\mathfrak S_{(\infty)}$ is a normal subgroup of its Schlichting completion with respect to $\Delta$.
\end{proposition}
\begin{proof}
By Theorem \ref{thm: classification maximal proper infinite index} we know that $\Delta$ is the stabilizer of a subset of $\N$ of cardinality $k$. Denote by $[\N]_k$ the set of subsets of $\N$ of cardinality $k$.  
Then $\mathfrak S_{(\infty)}$ acts transitively on $[\N]_k$, so we can identify $\mathfrak S_{(\infty)}/\Delta$ to $[\N]_k$ and the $\mathfrak S_{(\infty)}$-action on $\mathfrak S_{(\infty)}/\Delta$ extends to the natural $\mathfrak S_\infty$-action on $[\N]_k$. 
Denote by $\pi_k: \mathfrak S_\infty\to \mathfrak S([\N]_k)$ the associated continuous group morphism.  

Let us show $\pi_k$ is a homeomorphism onto its image. Suppose $\pi_k(\sigma_n)\to \id_{[\N]_k}$, and let $x\in \N$. We find $A,B\in [\N]_k$ such that $A\cap B=\{x\}$, so for large enough $n$ we have $\sigma_n(A)\cap \sigma_n(B)=A\cap B=\{x\}$ and we thus have $\sigma_n(x)=x$. We conclude $\sigma_n\to \id_\N$, so $\pi_k$ is indeed a homeomorphism onto its image.

We conclude $\pi_k(\mathfrak S_\infty)$ is a closed subgroup of $\mathfrak S([\N]_k)$ and thus the Schlichting completion of $\mathfrak S_{(\infty)}$ with respect to $\Delta$ is $\mathfrak S_\infty$.
\end{proof}

Let us make explicit the reasoning we used in the above proposition so as to identify a Schlichting completion. 

\begin{lemma}\label{lem: identifying Schlichting}
Let $A$ be a countable set, let $G\leq\mathfrak S(A)$ be a closed subgroup acting transitively on $A$. Suppose $\Gamma$ is a countable dense subgroup of $G$. Then for every $x\in A$, the Schlichting completion of $\Gamma$ with respect to $\Gamma_x$ is isomorphic to $G$ as a topological group. 
\end{lemma}
\begin{proof}
We have a natural identification of $A$ with $\Gamma/\Gamma_x$ which induces an isomorphism $\Phi:\mathfrak S(A)\to \mathfrak S(\Gamma/\Gamma_x)$. Notice that $\Phi(\Gamma)$ is then dense in $\Phi(G)$, and since the closure of $\Phi(\Gamma)$ is the Schlichting completion of $\Gamma$ with respect to $\Gamma_x$, we get that $\Phi(G)$ is the Schlichting completion of $\Gamma$ over $\Gamma_x$. 
\end{proof}

We now move on to case \eqref{eq:infinite biindex not transitive} from Theorem \ref{thm: classification maximal proper infinite index}. 
As it turns out, the study of this group as a Polish group was initiated by Cornulier \cite{Corn}. We need a few preliminary definitions. 

\begin{definition}
Let $X$ be an infinite coinfinite subset of $\N$.
The \textbf{commensurating subgroup} of $X$ is the group $\mathfrak S(\N,X)$ of all permutations $\sigma\in\mathfrak S_\infty$ such that $\sigma(X)\bigtriangleup X$ is finite. 
\end{definition}

The commensurating subgroup is a Polish non-archimedean group for the topology of pointwise convergence on the countable set $\Comm_X(\N)$ of subsets $Y$ of $\N$ \textbf{commensurated} to $X$, i.e. such that $X\bigtriangleup Y$ is finite. Note that $\mathfrak S(\N,X)$ acts transitively on $\Comm_X(\N)$. 

Define the \textbf{transfer character} $tr$ on $\mathfrak S(\N,X)$ by 
$$tr(\sigma)=\abs {\sigma X\setminus X}-\abs{X\setminus \sigma X}.$$
It can be shown that $tr$ is a continuous homomorphism $\mathfrak S(\N,X)\to \Z$ (see \cite[Prop. 4.H.1]{Corn}. Its kernel is denoted by $\mathfrak S^0(\N,X)$, it contains $\mathfrak S_{(\infty)}$ as a dense subgroup \cite[Prop. 4.H.3, Prop. 4.H.4]{Corn}. Moreover, $\mathfrak S^0(\N,X)$ acts transitively on $\Comm^0_X(\N):=\{Y\in\Comm_X(\N): \abs{X\setminus Y}=\abs{Y\setminus X}\}$. 
We have the following basic lemma on the relationship between $\Comm^0_X(\N)$ and $\Comm_X(\N)$.

\begin{lemma}
Every element of $\Comm_X(\N)$ can be written either as the reunion or the intersection of two elements of $\Comm^0_X(\N)$.
\end{lemma}
\begin{proof}
Let $Y\in \Comm_X(\N)$, we write $Y=(X\setminus F_1)\sqcup F_2$ with $F_1\Subset X$ and $F_2\Subset \N\setminus X$. If $\abs{F_1}>\abs{F_2}$, we find $F_3\subseteq \N\setminus (X\sqcup F_2)$ such that $\abs{F_3}=\abs{F_1}-\abs{F_2}$ and $F_4\subseteq F_1$ such that $\abs{F_4}=\abs{F_2}$.
Let $Y_1=(X\setminus F_1)\sqcup (F_2\sqcup F_3)$ and $Y_2= (X\setminus F_4)\sqcup F_2$. Then both $Y_1$ and $Y_2$ are in $\Comm_X^0(\N)$, and $Y=Y_1\cap Y_2$. The case $\abs{F_1}\geq\abs{F_2}$ is similar, and one obtains that $Y$ can be obtained as the union of two elements of $\Comm^0_X(\N)$. 
\end{proof}

\begin{proposition}\label{prop: completion maximal infinite biindex}
Let $X$ be an infinite coinfinite subset of $\N$. Denote by $\Delta$ the group of finitely supported permutations $\sigma$ such that $\sigma(X)=X$. Then the Schlichting completion of $\mathfrak S_{(\infty)}$ with respect to $\Delta$ is $\mathfrak S^0(\N,X)$. In particular, $\mathfrak S_{(\infty)}$ is a normal subgroup of its Schlichting completion with respect to $\Delta$.
\end{proposition}
\begin{proof}
Let $\pi$ be the continuous morphism obtained by restricting the $\mathfrak S^0(\N,X)$-action to $\Comm^0_X(\N)$.
By the previous lemma, if $\pi(\sigma_n)\to \mathrm{id}_{\Comm^0_X(\N)}$ then $\sigma_n\to \mathrm{id}_{\Comm_X(\N)}$ so $\pi$ is an embedding. 
So $\mathfrak S^0(\N,X)$ may be viewed as a closed subgroup of $\mathfrak S(\Comm^0_X(\N))$. Its action on $\Comm^0_X(\N)$ is moreover clearly transitive. Since $\mathfrak S^0(\N,X)$ contains $\mathfrak S_{(\infty)}$ as a dense subgroup, the action of $\mathfrak S_{(\infty)}$ on $\Comm^0_X(\N)$ is also transitive, and the stabilizer of $X$ for this action is our subgroup $\Delta$.
Since $\mathfrak S^0(\N,X)$ contains $\mathfrak S_{(\infty)}$ as a dense subgroup, the desired conclusion follows from Lemma \ref{lem: identifying Schlichting}.
\end{proof}

Our third last possible Schlichting completion (case \eqref{eq:infinite biindex transitive} from Theorem \ref{thm: classification maximal proper infinite index}) arises as follows. 
Let us say that an \textbf{enumerated partition} of $\N$ is a sequence 
$\mathcal P=(A_i)_{i\in\N}$ of disjoint subsets of $\N$ which cover $\N$.
Let $k\in\N \cup\{\aleph_0\}$ with $k\geq 2$, and consider an enumerated 
partition $\mathcal P=(A_i)_{i\in\N}$ of $\N$ where each $A_i$ has cardinality $k$. We then define the \textbf{almost full group} of the enumerated partition $\mathcal P$ as the group 
$$A[\mathcal P]=\{\sigma\in \mathfrak S_\infty: \sigma(A_i)=A_i\text{ for all but finitely many }i\in\N\}.$$

The topology on $A[\mathcal P]$ is obtained by declaring that a basis of neighborhoods of the identity is made of the neighborhoods of the identity of the \textbf{full group} of $\mathcal P$
$$[\mathcal P]=\{\sigma\in \mathfrak S_\infty: \sigma(A_i)=A_i\text{ for all }i\in\N\}$$
equipped with the topology induced by $\mathfrak S_\infty$. In particular, $[\mathcal P]$ is an open subgroup of $A[\mathcal P]$. 
\begin{lemma}
The topology obtained by declaring that a basis of neighborhoods of identity in $A[\mathcal P]$ is formed by neighborhoods of the identity in $[\mathcal P]$ does define a group topology on $A[\mathcal P]$, which is Polish and moreover locally compact when $k$ is finite.
\end{lemma}
\begin{proof}
To prove that we have a group topology, it suffices to show that for every $g\in A[\mathcal P]$ and every $V$ neighborhood of the identity in $[\mathcal P]$, the set $gVg\inv \cap V$ is still a neighborhood of the identity in $[\mathcal P]$. 
We may assume that $V$ is the pointwise stabilizer in $[\mathcal P]$ of a finite set $B_{i_1}\sqcup\cdots \sqcup B_{i_l}$, where $B_{i_j} \Subset A_{i_j}$ for all $j$. In other words $V$ is the set of all permutations which fix every point in $B_{i_1}\sqcup\cdots \sqcup B_{i_l}$ and leave every $A_i$ invariant. 
It is clear that $gVg\inv$ is the set of all permutations which fix every point in $gB_{i_1}\sqcup\cdots \sqcup gB_{i_l}$ and leave every $gA_i$ invariant. 
But $g\in A[\mathcal P]$ so there are only finitely many $A_i$'s such that $gA_i\neq A_i$. 
In particular, we can find $F\Subset \N$ containing  $\{i_1,\dots,i_l\}$ such that $gA_i=A_i$ for all $i\in \N\setminus F$ and $gB_{i_1}\sqcup\cdots \sqcup gB_{i_l}\subseteq \sqcup_{i\in F}A_i$.
Then $gVg\inv\cap V$ contains the set  $W$ of all permutations which fix pointwise every element of $B_{i_1}\sqcup\cdots \sqcup B_{i_l} \sqcup gB_{i_1}\sqcup\cdots \sqcup gB_{i_l} $ and leave every $A_i$ invariant. Since $W$ is a neighborhood of the identity in $[\mathcal P]$, this ends the proof that $gVg\inv \cap V$ is still a neighborhood of the identity in $[\mathcal P]$, and we conclude that our topology is a group topology. 

Now observe that since $[\mathcal P]$ is open of countable index in $A[\mathcal P]$, the associated topology is Polish because $A[\mathcal P]$ is homeomorphic to the Polish space
$\N\times [\mathcal P]$. 
The fact that the topology is locally compact when $k$ is finite follows from the fact that in this case $[\mathcal P]$ is compact, since it is a closed permutation group and all of its orbits are finite.  
\end{proof}

We now consider a more concrete way to view the topology we defined above 
so as to apply Lemma \ref{lem: identifying Schlichting}. 

Given a (possibly infinite, countable) $k \ge 2$ and $\mathcal P=(A_i)_{i\in\N}$ an enumerated partition of $\N$ into subsets of cardinality $k$, we denote $O(\mathcal P)$ the orbit of $\mathcal P$ under $\mathfrak S_{(\infty)}$ (for the natural action of $\mathfrak S_\infty$ on the set $\mathcal E_k$ of enumerated partitions of $\N$ into subsets of cardinality $k$). 

When $k$ is finite, one has $\sigma(O(\mathcal P))=O(\mathcal P)$ for any $\sigma \in A[\mathcal P]$; indeed, in that case $O(\mathcal P)$ is simply the set of all enumerated partitions of $\N$ whose elements all have cardinality $k$, and with only finitely many elements not equal to an element of $\mathcal P$. When $k$ is infinite, this is no longer true, as the orbit of $\mathcal P$ under the action of $A[\mathcal P]$ is uncountable. We denote 
$$\tilde A[\mathcal P]= \{\sigma \in A[\mathcal P] \colon \sigma(O(\mathcal P))= O(\mathcal P)\}$$

When $k$ is finite, we simply have $\tilde A[\mathcal P]=A[\mathcal P]$

\begin{proposition}
$\tilde A[\mathcal P]$ is closed in $A[\mathcal P]$.
\end{proposition}

Hence $\tilde A[\mathcal P]$ is a Polish group in its own right. Before giving the proof, we would like to note that our original argument was flawed; the correct proof below is due to the referee. The proof itself is fairly routine, but we feel like it is worth mentioning the unusual level of care and precision of the referee's work (which we could also have pointed out in several other places) and express our gratitude again.

\begin{proof}
When $k$ is finite there is nothing to prove. Assume $k=\aleph_0$, and let $O'(\mathcal P)$ denote the orbit of $\mathcal P$ under the action of $A[\mathcal P]$; explicitly, $O'(\mathcal P)$ consists of all enumerated partitions $\mathcal Q=(B_i)_{i \in \N}$ with $|B_i|=\aleph_0$ for all $i$, and $B_i=A_i$ for all but finitely many $i$.

Since $[\mathcal P]$ is the stabilizer of $\mathcal P$ for the action of $A[\mathcal P]$ on $O'(\mathcal P)$, and $[\mathcal P]$ is an open subgroup of $A[\mathcal P]$, the morphism $\pi \colon \tilde A[\mathcal P] \to \mathfrak S(O'(\mathcal P))$ is continuous (as usual $\mathfrak S(O'(\mathcal P))$ is endowed with the topology of pointwise convergence for the discrete topology on $O'(\mathcal P)$; one should perhaps note here that $O'(\mathcal P)$ is not countable, but that does not affect the argument). 

Observe that $\tilde A[\mathcal P]= \{\sigma \in A[\mathcal P] \colon \sigma(O(\mathcal P))= O(\mathcal P)\}$, which proves that $\tilde A[\mathcal P]$ is closed in $A[\mathcal P]$. 

\end{proof}

\begin{proposition}
$\mathfrak S_{(\infty)}$ is dense in $\tilde A[\mathcal P]$.
\end{proposition}

\begin{proof}
Pick $g \in \tilde A[\mathcal P]$. We may find $\sigma \in \mathfrak S_{(\infty)}$ such that $\sigma(\mathcal P)=g(\mathcal P)$, hence $\sigma^{-1}g \in [\mathcal P]$. The desired result now follows from the density of the group of finitely supported elements of $[\mathcal P]$ in $[\mathcal P]$.
\end{proof}

\begin{proposition}
The topology on $\tilde A[\mathcal P]$ coincides with the topology of pointwise convergence on $O(\mathcal P)$. 
\end{proposition}
\begin{proof}
The natural map $\pi: \tilde A[\mathcal P] \to \mathfrak S(O(\mathcal P))$ is continuous (as above, this follows from the fact that $[\mathcal P]$ is the stabilizer of $\mathcal P$ for the action of $\mathfrak S_\infty$ on $\mathcal E_k$ and $[\mathcal P] \cap \tilde A[\mathcal P]$ is open in $\tilde A[\mathcal P]$). 

Note next that $\pi$ is injective. To see this, observe that any element of the kernel of $\pi$ must belong to $[\mathcal Q]$ for any $\mathcal Q \in O(\mathcal P)$. Pick such a $\sigma \in \tilde A[\mathcal P]$.

Then let $x\in\N$, write again $\mathcal P=(A_i)_{i\in\N}$ and let $i_0\in\N$ be such that $x\in A_{i_0}$. Let $j_0\neq i_0$, fix $y\in A_{j_0}$. 
Define a new enumerated partition $\mathcal Q=(B_i)_{i\in\N}$ by letting 
\[
B_i=\left\{\begin{array}{cl}
\{x\}\cup (A_{j_0}\setminus \{y\}) & \text{if }i=i_0 \\
\{y\}\cup (A_{i_0}\setminus \{x\}) & \text{if }i=j_0 \\
A_i & \text{otherwise.}
\end{array}\right.
\]
Observe that $\mathcal Q \in O(\mathcal P)$ and, since $\sigma$ belongs to both $[\mathcal Q]$ and $[\mathcal P]$, it must fix $A_{i_0}\cap B_{i_0}=\{x\}$, so $\sigma(x)=x$. We conclude that $\sigma=1$, so that $\pi$ is injective as promised.

Suppose that $(\sigma_n) \in \tilde A[\mathcal P]$ is such that $\pi(\sigma_n)$ converges to $1$ in $\mathfrak S(O(\mathcal P))$; pick $x \in \N$ and define $\mathcal Q \in O(\mathcal P)$ in the same way as above. For $n$ large enough, one must have $\sigma(\mathcal P)=\mathcal P$ and $\sigma(\mathcal Q)=\mathcal Q$, hence $\sigma(x)=x$. This proves that $\pi$ is a topological group embedding.
\end{proof}

\begin{corollary}\label{cor: completion infinite biindex transitive}
Let $H_k$ be the group of finitely supported bijections which fix an enumerated 
partition $\mathcal P_k$ of $\N$ into sets of (possibly infinite) cardinality $k$. 
Recall that $\tilde A[\mathcal P_k]$ denotes the closure of $\mathfrak S_{(\infty)}$ in $A[\mathcal P_k]$.
 
Then the Schlichting completion of $\mathfrak S_{(\infty)}$ with respect to $H_k$ is equal to $\tilde A[\mathcal P_k]$.
In particular, $\mathfrak S_{(\infty)}$ is a normal subgroup of its Schlichting completion with respect to $H_k$.

\end{corollary}
\begin{proof}
This follows from the two previous propositions, combined with Lemma \ref{lem: identifying Schlichting}.
\end{proof}

\subsection{Proof of Theorem \ref{t:anexample}}

Combining Proposition \ref{prop: completion maximal finite biindex}, Proposition \ref{prop: completion maximal infinite biindex} and Corollary \ref{cor: completion infinite biindex transitive}  with the classification of the maximal proper subgroups of $\mathfrak S_{(\infty)}$ of infinite index, we obtain the following proposition. 

\begin{proposition}
Let $\Delta\leq\mathfrak S_{(\infty)}$ be a maximal proper subgroup of $\mathfrak S_{(\infty)}$ of infinite index. Then $\mathfrak S_{(\infty)}$ is a normal subgroup of its Schlichting completion with respect to $\Delta$.
\end{proposition}

\begin{question}
Is there a way to establish that result more directly, without first having to describe the maximal proper subgroups? In the case of $\mathfrak S_{(\infty)}$ that was not hard to do (and is sufficient to establish the result above, since then the homogeneous spaces $\mathfrak S_{(\infty)}/\Delta$ are easy to understand) but that sort of hands-on approach seems unlikely to work in more general settings. 
\end{question}

Recall from Proposition \ref{prop: dense in iso implies maximal proper} that if $\mathfrak S_{(\infty)}\leq \Iso(\US)$ is dense, then the $\mathfrak S_{(\infty)}$-stabilizer of any point in $\US$ is a maximal proper subgroup of infinite index. As we observed before, $\Iso(\US)$ must then be isomorphic to the Schlichting completion of $\mathfrak S_{(\infty)}$ with respect to that subgroup. 
To finish the proof of Theorem \ref{t:anexample}, it is thus enough to prove the following fact.

\begin{proposition}
Assume that $|S| \ge 3$. Then the conjugacy class of any nonidentity element of $\Iso(\U_S)$ is uncountable. In particular, $\Iso(\U_S)$ does not admit a nontrivial countable normal subgroup.
\end{proposition}

\begin{proof}
Assume that $g \in \Iso(\U_S)$ has a countable conjugacy class; then its centralizer $C(g)$ is a closed subgroup with countable index, so it is clopen. Hence, there must exist a finite subset $F$ of $\U_S$ such that $g$ commutes with any element $h$ satisfying $h(x)=x$ for all $x \in F$. We want to prove that $g=\id$ and assume for a contradiction that such is not the case.

We first claim that there must exist $x \not \in F$ such that $g(x) \neq x$. To see this, pick $y$ such that $g(y) \ne y$. Since $|S| \ge 3$, we can find $s_1 \ne s_2 \in S$ such that the map $f \colon \{y,g(y)\} \to \R$ defined by setting $f(y)=s_1$, $f(g(y))=s_2$ is Kat\v{e}tov. Thus there exist infinitely many $x \in \U_S$ such that $d(x,y)=s_1$ and $d(x,g(y))=s_2$; in particular there exists such an $x$ which does not belong to $F$ and since $d(x,y) \ne d(x,g(y))$ we have $g(x) \ne x$.

Next, pick $h \in \Iso(\U_S)$ such that $h$ coincides with the identity on $F \cup \{x\}$ yet $h(g(x))\neq g(x)$. Then $hg(x) \neq gh(x)$, so $h$ does not commute with $g$, a contradiction. Hence $g=\id$.
\end{proof}

Another nice consequence of Theorem \ref{thm: classification maximal proper infinite index} is that $\mathfrak S_{(\infty)}$ has only one $2$-transitive action on $\N$ up to conjugacy. 
\begin{proposition}
Let $\mathfrak S_{(\infty)}\act\N$ be a $2$-transitive action. Then the action is conjugate to the natural action of $\mathfrak S_{(\infty)}$ on $\N$. 
\end{proposition}
\begin{proof}
Let $\Delta=\Stab_{\mathfrak S_{(\infty)}}(0)$, then $\Delta$ has infinite index and by $2$-transitivity its bi-index is $2$. This implies that $\Delta$ is maximal proper: if $\Delta<\Gamma$ then $\Gamma$ must contain a $\Delta$-double coset distinct from $\Delta$ and is thus equal to $\mathfrak S_{(\infty)}$.  The desired conclusion now follows from Theorem \ref{thm: classification maximal proper infinite index}.
\end{proof}
\begin{remark}
In \cite{Maitre:2015wd}, the second author defines a transitive action of a countable group $\Gamma$ on a countable set $X$ to be \emph{highly faithful}  if for every $F\Subset\Gamma$, there is $x\in X$ such that for all distinct $f_1,f_2\in F$, $f_1\cdot x\neq f_2\cdot x$. The natural action of $\mathfrak S_{(\infty)}$ on $\N$ is highly transitive but not highly faithful, so the above proposition yields that $\mathfrak S_{(\infty)}$ has no highly transitive highly faithful action, answering a question raised in \cite{Maitre:2015wd}.
\end{remark}

\bibliographystyle{alpha}
\bibliography{biblio2}

\end{document}